\documentclass[12pt,a4paper]{amsart}
\usepackage{amsmath}
\usepackage{paralist}
\usepackage{graphics} 
\usepackage{epsfig} 
\usepackage{graphicx}  
\usepackage{epstopdf}

\newtheorem{theorem}{Theorem}[section]

\theoremstyle{definition}
\newtheorem{definition}[theorem]{Definition}
\newtheorem{remark}{Remark}

\newcommand{\hu}{\hat{u}}
\newcommand{\hx}{\hat{x}}
\newcommand{\hatt}{\hat{t}}
\newcommand{\heta}{\hat{\eta}}
\newcommand{\hv}{\hat{v}}
\newcommand{\hlambda}{\hat{\lambda}}

\newcommand{\hT}{\hat{T}}
\newcommand{\hP}{\hat{P}}
\newcommand{\unit}[1]{\mathrm{#1}}
\usepackage{amsmath,amssymb,amsfonts}
\usepackage[colorinlistoftodos]{todonotes}

\usepackage{graphicx}
\usepackage{color}
\usepackage{bm}

\graphicspath{{./figures_arxiv/}}

\title[The role of viscoelasticity in deformable porous media]
{The role of structural viscoelasticity in deformable porous media with incompressible
constituents: applications in biomechanics}

\address{$^{1}$ Dipartimento di Matematica, Politecnico di Milano,  
				 Piazza L. da Vinci 32, 20133 Milano, Italy}
\address{$^{(2)}$ Department of Electrical Engineering and Computer Science, 
College of Engineering, University of Missouri, 201 Naka Hall, Columbia, MO 65211}
\address{$^{(3)}$ NC State University, Department of Mathematics, SAS Hall 3236, 
Raleigh, NC, 27695, USA}

\author{Maurizio Verri$^{1}$ \and Giovanna Guidoboni$^{2}$ \and Lorena Bociu$^{3}$
\and Riccardo Sacco$^{1}$}
\email{maurizio.verri@polimi.it}
\email{guidobonig@missouri.edu}
\email{lvbociu@ncsu.edu}
\email{riccardo.sacco@polimi.it}

\begin{document}

\date{\today}

\begin{abstract}
The main goal of this work is to clarify and quantify, by means of mathematical analysis, the role of structural viscoelasticity in the biomechanical response of deformable porous media 
with incompressible constituents to sudden changes in external applied loads. 
Models of deformable porous media with incompressible constituents
are often utilized to describe the behavior of biological tissues, such as cartilages, bones and engineered tissue scaffolds, where viscoelastic properties may change with age, disease or by design. Here, for the first time, we show that the fluid velocity within the medium could increase tremendously, even up to infinity, should the external applied load experience sudden changes in time and the structural viscoelasticity be too small. In particular, we consider a one-dimensional poro-visco-elastic model
for which we derive explicit solutions in the cases where the external applied load is characterized by a step pulse or a trapezoidal pulse in time. By means of dimensional analysis, we identify some dimensionless parameters that can aid the design of structural properties and/or experimental conditions as to ensure that the fluid velocity within the medium remains bounded below a certain given threshold, thereby preventing potential tissue damage. The application to confined compression tests for biological tissues is discussed in detail. Interestingly, the loss of viscoelastic tissue properties has been associated with various disease conditions, such as atherosclerosis, Alzheimer's disease and glaucoma. Thus, the findings of this work may be relevant to many applications in biology and medicine.
\end{abstract}

\maketitle

{\bf Keywords:}
Deformable porous media flow; incompressible constituents; viscoelasticity; 
explicit solution; velocity blow-up; confined compression.

\section{Introduction}
\label{sec:intro}

Fluid flow through deformable porous  media is relevant for many applications in
biology, medicine and bioengineering. Some important examples include blood flow through tissues in the human body~\cite{chapelle,causin}  and fluid flow inside cartilages, bones and engineered tissue scaffolds~\cite{cowin,lai,Soltz1998,Sacco2011,Causin2013}.
The mechanics of biological tissues typically exhibits both elastic and viscoelastic behaviors resulting from the combined action of various components, including elastin, collagen and extracellular
matrix~\cite{mow,nia,Ozkaya,RechaSancho2016}.
Thus, from the mathematical viewpoint, the study of fluid flows through deformable porous  biological structures requires the coupling of poro-elasticity with structural viscoelasticity, leading to \textit{poro-visco-elastic models}.

The theoretical study of fluid flow through deformable porous media has attracted a lot of attention since the beginning of the last century, initially motivated by applications in geophysics and petroleum engineering. The development of the field started with the work of Terzaghi in 1925~\cite{terzaghi}, which focused on finding an analytic solution for a one-dimensional (1D) model. However, it was Biot's work in 1941~\cite{biot} that set up the framework and ignited the mathematical development for fluid flow through poro-elastic media. To date, several books and articles have been devoted to the mathematical analysis and numerical investigation of poro-elastic models, such as \cite{Soltz1998,deBoer,coussy,detournay1,detournay2,zenisek,owczarek,show1,show2,cao,phillips,phillips2,phillips3,settari},
with  applications ranging from engineering and geophysics to medicine and biology.
Recently, our team has developed a theoretical and numerical framework to study both poro-elastic and poro-visco-elastic models, as
motivated by biological applications~\cite{ARMA}. The study showed that structural viscoelasticity plays a crucial role in determining the regularity requirements for volumetric and boundary forcing terms, as well as for the  corresponding solutions. Moreover, in~\cite{BBBNG} it has been shown that the solution of the fluid-solid mixture (elastic displacement, fluid pressure, and Darcy velocity) is more sensitive to the boundary traction in the elastic case than in the visco-elastic scenario. These theoretical findings are
also supported by experimental and clinical evidences showing that changes in tissue viscoelasticity are associated with various pathological conditions, including atherosclerosis~\cite{osidak}, osteoporosis~\cite{augat}, renal disease~\cite{guerin} and glaucoma~\cite{downs}.

Interestingly, the study in~\cite{ARMA} provided numerical clues that sudden changes in body forces and/or stress boundary conditions may lead to uncontrolled fluid-dynamical responses within the medium in the absence of structural viscoelasticity. This finding led us to formulate a \textit{novel hypothesis} concerning the causes of damage in biological tissues, namely
that \textit{abrupt time variations in stress conditions combined with lack of structural viscoelasticity could lead to microstructural damage due to local fluid-dynamical alterations}, as illustrated in Fig.~\ref{fig:hypothesis}.

\vspace{.2cm}

\begin{figure}[h!]
\begin{center}
 \includegraphics[scale=0.45]{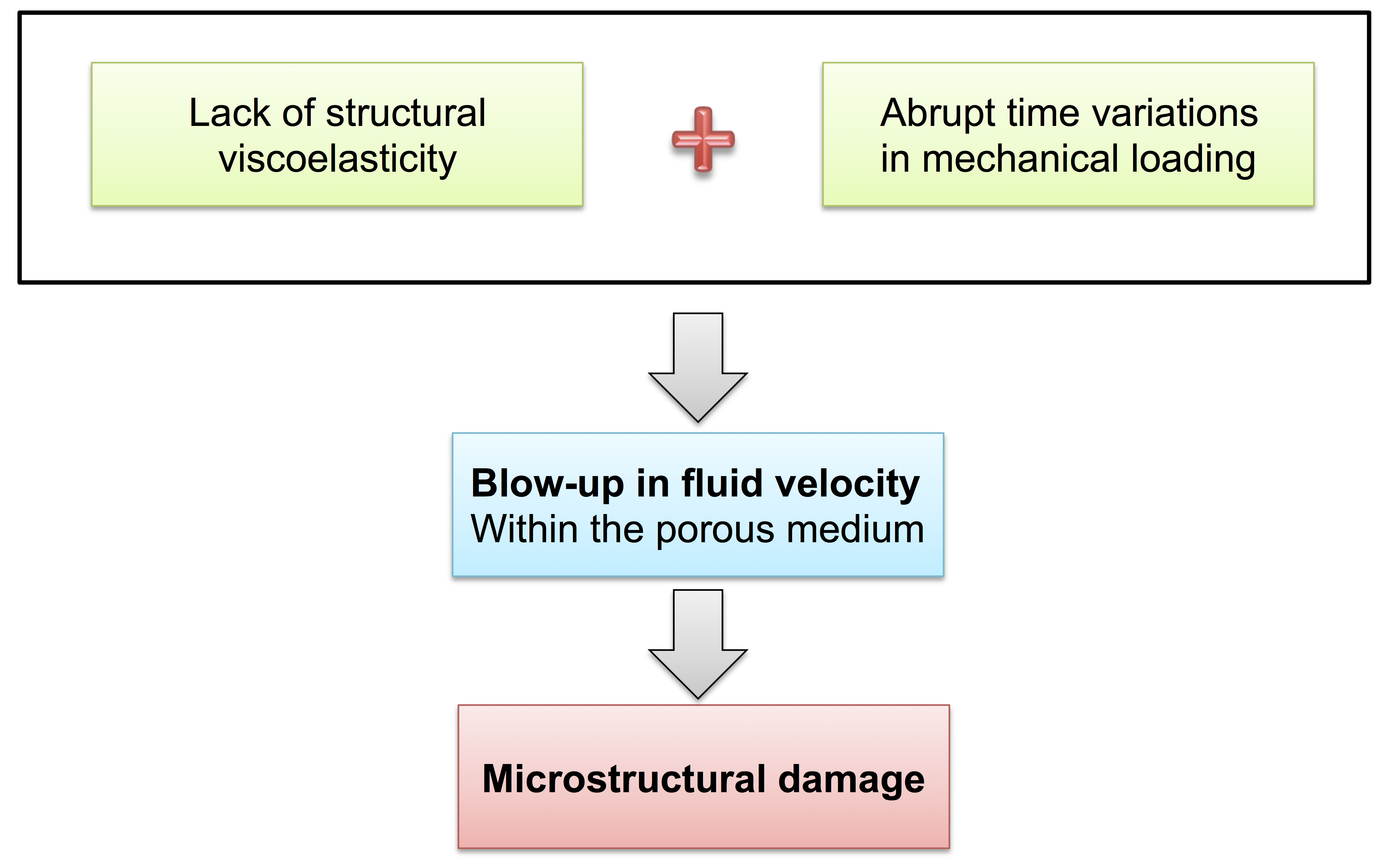}
\end{center}
\caption{Illustration of a novel hypothesis on potential causes of microstructure damage in deformable porous media. In the case that (i) the medium is subjected to a time-discontinuous mechanical load and (ii) structural viscoelasticity is reduced or absent, then the fluid velocity within the porous medium will experience a blow-up, possibly leading to microstructural damage.}\label{fig:hypothesis}
\end{figure}

The importance of the coupling between structural mechanics and fluid dynamics in the damage of deformable porous media has been investigated by several authors~\cite{nia,mahyar,selvadurai}. In the present study, we focus on a particular aspect of this coupling and we aim at characterizing and quantifying the influence of structural viscoelasticity on the biomechanical and fluid-dynamical responses to sudden changes in stress conditions. 
Biomechanical applications are characterized by the fact that tissues
have a mass density that is similar to that of water. 
For this reason, we consider in these pages 
the case of deformable porous media constituted by \emph{incompressible solid and fluid components.} 
To mathematically define this concept, we introduce the following 
relation between fluid pressure $p$, variation of fluid content $\zeta$ and volumetric dilation 
$\nabla \cdot \mathbf{u}$, $\mathbf{u}$ denoting the solid displacement
\begin{subequations}\label{eq:zeta}
\begin{equation}\label{eq:fluid_content}
\zeta = c_0 p + \alpha \nabla \cdot \mathbf{u},
\end{equation}
where $c_0$ is the storage coefficient and $\alpha$ is the Biot coefficient (see~\cite{detournay2}).
In the case of incompressible solid and fluid components, we have $c_0=0$ and $\alpha=1$ 
(see~\cite{detournay2}) so that~\eqref{eq:fluid_content} becomes
\begin{equation}\label{eq:fluid_content_incompressible}
\zeta = \nabla \cdot \mathbf{u}.
\end{equation}
\end{subequations}
Notice that unlike in standard elasticity theory, incompressibility of each component of a deformable 
porous medium \emph{does not} mean that \emph{both} solid displacement and fluid velocity are divergence-free, rather,
that the volumetric deformation of the solid constituent corresponds to the 
variation of fluid volume per unit volume of porous material, with the convention that
$\nabla \cdot \mathbf{u}$ is positive for extension while $\zeta$ is positive for a 
"gain" of fluid by the porous solid.

Building upon our theoretical and numerical results presented in~\cite{ARMA}, we devise a 1D problem  for which
we exhibit an explicit solution where the discharge velocity goes to infinity if the stress boundary condition is not sufficiently smooth in time and the solid component is not viscoelastic. Interestingly, this blow-up in the velocity occurs even in the simple case where the permeability of the medium is assumed to be constant, in comparison to the general case of nonlinear permeability depending on dilation \cite{cao,ARMA} or pressure \cite{show2}. In addition, we perform a dimensional analysis that allows us to identify the parameters influencing the solution blow-up, thereby opening the path to sensitivity analysis on the system, and providing practical directions on how to control the biomechanical and fluid-dynamical response of the fluid-solid mixture, prevent microstructural damage and, in perspective, aid the experimental design of bioengineered tissues~\cite{Freeman2009}.

The article is organized as follows. Poro-visco-elastic models are described in Section~\ref{sec:model}, along with a summary of the related theoretical results. Section~\ref{sec:1D_problem_and_formulation}
focuses on a special 1D case, for which an explicit solution is derived and its well-posedness is studied in
the presence or absence of viscoelasticity. The dimensional analysis of the 1D problem is carried out in Section~\ref{sec:dimensionless_problem}. Solution properties are explored in detail for the particular case of boundary traction with a discontinuity in time, see Section~\ref{sec:step_pulse}, and with a trapezoidal time profile, see Section~\ref{sec:trapezoidal_pulse}. The application of this analysis to the case of confined compression of biological tissues is discussed in Section~\ref{sec:athesian_soltz_case}. Conclusions and perspectives are summarized in Section~\ref{sec:conclusions}.

\section{Poro-Visco-Elastic Models: description and theoretical results}
\label{sec:model}

Following the same notation as in~\cite{ARMA}, let $\Omega \subset \mathbb R^d$, with $d=1,2$ or 3, be the region of space occupied by a deformable porous medium and let $(0,T)$ be the observational time interval.
In the assumptions of small deformations, full saturation of the mixture, incompressibility of the mixture 
constituents and negligible inertia, we can write the balance of linear momentum for the mixture and the balance of mass for the fluid component as
\begin{subequations}\label{eq:big_model}
\begin{align}
\nabla \cdot \bm{\sigma} + \mathbf{F} = \mathbf 0 && \mbox{in}\; \Omega\times (0,T)\label{eq:big_model_in}\\
\dfrac{\partial \zeta}{\partial t} + \nabla \cdot \mathbf v = S && \mbox{in}\; \Omega\times (0,T)
\end{align}
respectively, where $\bm{\sigma}$ is the total stress tensor, $\zeta$ is the fluid content, $\mathbf v$ is the discharge velocity and $\mathbf{F} $ and $S$ are volumetric sources of linear momentum and fluid mass, respectively.
The system must be completed by constitutive equations for $\bm{\sigma}$, $\zeta$ and $\mathbf v$. In line with the literature for poro-visco-elastic models in biological applications, we will adopt the following constitutive equations:
\begin{align}
\bm{\sigma} &= \bm{\sigma}_{e} + \bm{\sigma}_{v} - p \mathbf{I}\\
  \bm{\sigma}_{e} &= \lambda_e (\nabla \cdot \mathbf u )\, \mathbf I + 2\mu_e \bm{\epsilon}(\mathbf u)\\
    \bm{\sigma}_{v} &= \lambda_v \left(\nabla \cdot  \dfrac{\partial \mathbf u}{\partial t} 
		\right)\, \mathbf I + 2\mu_v \bm{\epsilon}\left(\dfrac{\partial \mathbf u}{\partial t} \right)\label{eq:sigma_v}\\
    \zeta & = \nabla \cdot \mathbf u \label{eq:def_zeta} \\
    \mathbf v &= -K \nabla p \label{eq:big_model_end}
\end{align}
where
$\mathbf u$ is the displacement of the solid component, $p$ is the Darcy pressure,
$\mathbf{I}$ is the identity tensor and $\bm \epsilon (\mathbf w)=(\nabla \mathbf w + (\nabla \mathbf w)^T)/2$ is the symmetric part of the gradient of the vector field $\mathbf w$. The elastic parameters $\lambda_e$ and $\mu_e$  and the viscoelastic parameters  $\lambda_v$ and $\mu_v$ are assumed to be given positive constants.
We remark that structural viscoelasticity is embodied in the term $\bm{\sigma}_v$ defined in~\eqref{eq:sigma_v}. If structural viscoelasticity is not present, then $\bm{\sigma}_v = \mathbf 0$ and the set of equations~\eqref{eq:big_model_in}-\eqref{eq:big_model_end} defines a \textit{poro-elastic medium}.
In addition, the permeability $K$ may be a positive constant, as in \cite{biot}, or may depend nonlinearly on the solution, for example on the structural dilation, namely $K=K(\nabla \cdot \mathbf u)$, as in~\cite{cao,ARMA}. 
Finally, notice that~\eqref{eq:def_zeta} corresponds to~\eqref{eq:fluid_content_incompressible} 
consistently with the assumption of incompressibility of mixture constituents.
\end{subequations}

Most of the theoretical studies focused on the poro-elastic case without accounting for structural viscoelasticity. In  the case of constant permeability, the coupling between the elastic and fluid subproblems is linear and the well-posedness and regularity of solutions have been studied by several authors \cite{zenisek, owczarek,show1}.
In the case of non-constant permeability, the coupling between the two subproblems becomes nonlinear and only few theoretical results have been obtained. In \cite{show2}, Showalter utilized monotone operator theory techniques in order to provide well-posedness of solutions in the case where the permeability is a nonlinear function of pressure.
%
%
%
To the best of our knowledge, Cao et al in \cite{cao} were the first to consider the permeability as a nonlinear function of dilation and provide existence of weak solutions for this nonlinear poro-elastic case. However, the analysis in \cite{cao} is performed upon assuming 
homogeneous boundary conditions for both pressure and elastic displacement, which is often not the case from the viewpoint of applications.

Our recent paper in \cite{ARMA} extends  the works mentioned above by considering
poro-elastic and poro-visco-elastic models with dilation-dependent permeability, non-zero volumetric sources of mass and momentum and non-homogeneous, mixed Dirichlet-Neumann boundary conditions. More precisely, in \cite{ARMA} we assumed that  the boundary of $\Omega$ can be decomposed as $\partial \Omega = \Gamma_N \cup \Gamma_D$, with $\Gamma_D = \Gamma_{D,p}\cup \Gamma_{D,v}$, where the following boundary conditions are imposed:
\begin{align}
\bm\sigma \cdot \mathbf n = \mathbf g, \quad \mathbf v \cdot \mathbf n = 0 && \mbox{on}\; \Gamma_N\times (0,T)\\
\mathbf u  = \mathbf 0, \quad p = 0 && \mbox{on}\; \Gamma_{D_p}\times (0,T)\\
\mathbf u  = \mathbf 0, \quad \mathbf v \cdot \mathbf n = \psi && \mbox{on}\; \Gamma_{D_v}\times (0,T).
\end{align}
Our analysis showed that the data time regularity requirements and the smoothness of solutions significantly differ depending on whether the model is poro-elastic or poro-visco-elastic. In particular, in the visco-elastic case, if the source of linear momentum $\mathbf F$ is $L^2$ in time and space, namely $\mathbf F\in L^2(0,T;(L^2(\Omega))^d)$, and the source of boundary traction $\mathbf g$ is $L^2$ in time and $H^{1/2}$ on the boundary, namely $\mathbf g \in L^2(0,T;(H^{1/2}(\Gamma_N))^d)$,
 then there exists a weak solution to the system, with elastic displacement  $\mathbf u \in H^1$ in time and space, namely $\mathbf u \in H^1(0,T;(H^1(\Omega))^d)$, and pressure  $p \in L^2(0,T;H^1(\Omega)) $.
 In comparison, in the poro-elastic case, one needs both $\mathbf F$ and $\mathbf g$ to be $H^1$ in time in order to obtain a solution where both $\mathbf u$ and $p$ are $L^2$ in time and $H^1$ in space. This proves that the system dynamics fundamentally changes as visco-elastic effects fade away. Moreover, the numerical simulations performed in \cite{ARMA} hinted at blow-up of fluid energy functional in the poro-elastic case with non-smooth sources $\mathbf F$ and $\mathbf g$. In the following sections, we are going to further elaborate these concepts by means of a 1D poro-visco-elastic model for which we can obtain an explicit solution and assess the effect of structural viscoelasticity on the response of the deformable porous medium to sudden changes in mechanical load. The 1D model is described next.

\section{1D poro-visco-elastic model: description and formulation}\label{sec:1D_problem_and_formulation}

Let the space domain $\Omega$ be given by the
interval $(0,L)$ and the volumetric sources of linear momentum and fluid
mass be zero. In this case, the balance of linear momentum for the mixture
and the balance of mass for the fluid component simplify to
\begin{subequations}\label{eq:1D_model}
\begin{align}
\dfrac{\partial \sigma }{\partial x}=0&  & & \mbox{in}\;(0,L)\times (0,T)
\label{eq:1D_model_in} \\
\dfrac{\partial ^{2}u}{\partial t\partial x}+\dfrac{\partial v}{\partial x}%
=0&  & & \mbox{in}\;(0,L)\times (0,T).  \label{eq:1D_model_in2}
\end{align}%
\end{subequations}
The associated constitutive equations are given by
\begin{subequations}\label{eq:1D_constitutive_laws}
\begin{align}
\sigma _{0}=\mu \dfrac{\partial u}{\partial x}+\eta \dfrac{\partial ^{2}u}{%
\partial t\partial x}&  & & \mbox{in}\;(0,L)\times (0,T) \label{eq:1D_model_sigma_0} \\
\sigma =\sigma _{0}-p&  & & \mbox{in}\;(0,L)\times (0,T)
\label{eq:1D_model_sigma} \\
v=-K\dfrac{\partial p}{\partial x}&  & & \mbox{in}\;(0,L)\times (0,T)
\label{eq:1D_model_v}
\end{align}%
\end{subequations}
where we have set $\mu =\lambda _{e}+2\mu _{e}$, $\eta =\lambda
_{v}+2\mu _{v}$ and we have introduced the symbol $\sigma_0$ to denote the contribution of the solid component to the total stress tensor $\sigma$. We emphasize that the permeability may be a function of space and dilation, namely
$$
K=K\left(x,\frac{\partial u}{\partial x}\right).
$$
We complete the system with the
following boundary and initial conditions:
\begin{subequations}\label{eq:1D_BC_IC_s}
\begin{align}
u\left( 0,t\right) =v\left( 0,t\right) =0&  & & \mbox{for}\;0<t<T
\label{eq:bottom_bc} \\
p\left( L,t\right) =0\ ;\ \sigma \left( L,t\right) =-P\left( t\right) &  & & %
\mbox{for}\;0<t<T  \label{eq:top_bc} \\
u\left( x,0\right) =0&  & & \mbox{for}\;0<x<L  \label{eq:1D_model_end}
\end{align}%
\end{subequations}
where $g(t) = -P(t)$ is the boundary traction. For the ease of reference, the boundary conditions are schematized in Fig.~\ref{fig:1D_model}.
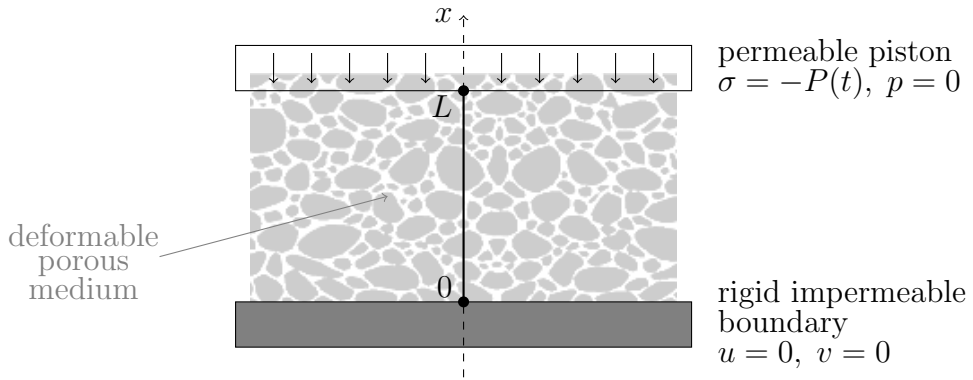
\begin{figure}[h!]
\begin{center}
\begin{tikzpicture}
\node[inner sep=0pt] at (0,0)
    {\includegraphics[width=.4\textwidth]{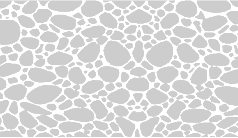}};
\draw [fill=gray] (-3,-2) rectangle (3,-1.4);
\node[right] at (3.2,-1.3) {rigid impermeable};
\node[right] at (3.2,-1.7) {boundary};
\node[right] at (3.2,-2.1) {$u=0, \; v=0$};
\draw (-3,2) rectangle (3,1.4);
\draw[->,black] (-2.5,1.9) -- (-2.5,1.5);
\draw[->,black] (-2,1.9) -- (-2,1.5);
\draw[->,black] (-1.5,1.9) -- (-1.5,1.5);
\draw[->,black] (-1,1.9) -- (-1,1.5);
\draw[->,black] (-0.5,1.9) -- (-0.5,1.5);
\draw[->,black] (2.5,1.9) -- (2.5,1.5);
\draw[->,black] (2,1.9) -- (2,1.5);
\draw[->,black] (1.5,1.9) -- (1.5,1.5);
\draw[->,black] (1,1.9) -- (1,1.5);
\draw[->,black] (0.5,1.9) -- (0.5,1.5);
\node[right] at (3.2,1.9) {permeable piston};
\node[right] at (3.2,1.5) {$\sigma=-P(t), \; p=0$};
\draw[black, thick] (0,-1.4) -- (0,1.4);
\draw[->,black, dashed] (0,1.4) -- (0,2.4);
\node[left, black] at (0,2.4) {$x$};
\draw[black, dashed] (0,-2.4) -- (0,-1.4);
\fill (0,1.4) circle (.4ex);
\fill (0,-1.4) circle (.4ex);
\node[left] at (0,-1.2) {$0$};
\node[left] at (0,1.2) {$L$};
\draw[->,gray] (-4,-0.8) -- (-1,0);
\node[above,gray] at (-5,-0.8) {deformable};
\node[above,gray] at (-5,-1.2) {porous};
\node[above,gray] at (-5,-1.5) {medium};
\end{tikzpicture}
\end{center}
\caption{Schematic representation of the one-dimensional problem considered in this article. The deformable porous medium may be either poro-elastic or poro-visco-elastic. The forcing term $P(t)$ may have discontinuities in time.}
\label{fig:1D_model}
\end{figure}

If $\eta =0$ (i.e., no viscoelasticity is present) and under the assumption
that the permeability $K$ and the boundary forcing term $P$ are given positive constants, system~\eqref{eq:1D_model}
degenerates into the poro-elastic model studied experimentally and
analytically in~\cite{Soltz1998}. Specifically, the boundary and initial
conditions~\eqref{eq:bottom_bc}-\eqref{eq:1D_model_end} correspond to the
creep experimental test performed in confined compression, where the boundary condition~\eqref{eq:bottom_bc}
represents an impermeable interface at the bottom of the specimen whereas
the boundary condition~\eqref{eq:top_bc} represents a free-draining permeable piston
at the top of the specimen.

System~\eqref{eq:1D_model_in}-\eqref{eq:1D_model_end} can be rewritten solely in
terms of the displacement. Indeed, integration of (\ref{eq:1D_model_in2}%
) with respect to $x$ yields
\begin{equation*}
\dfrac{\partial u}{\partial t}(x,t)+v(x,t)=A(t)
\end{equation*}%
where $A\left( t\right) $ is arbitrary. Thus, taking \eqref{eq:bottom_bc} into
account, we obtain that $A\left( t\right) =0$. From (\ref{eq:1D_model_v}), we then have that%
\begin{equation*}
\dfrac{\partial u}{\partial t}-K\dfrac{\partial p}{\partial x}=0\quad \text{%
in}\ (0,L)\times (0,T)\,.
\end{equation*}%
On the other hand, using (\ref{eq:1D_model_sigma}) and (\ref{eq:1D_model_in}),
we derive that
\begin{equation*}
\dfrac{\partial p}{\partial x}=\dfrac{\partial \sigma _{0}}{\partial x}\,.
\end{equation*}%
Therefore the system \eqref{eq:1D_model_in}-\eqref{eq:1D_model_end} reduces to the following
initial boundary value problem in terms of the sole elastic displacement $u$:
\begin{subequations}
\label{eq:1d_problem_per_u}
\begin{align}
\dfrac{\partial u}{\partial t}-K\mu \dfrac{\partial ^{2}u}{\partial x^{2}}%
-K\eta \dfrac{\partial ^{3}u}{\partial t\partial x^{2}}=0 &  & & %
\mbox{in}\ (0,L)\times (0,T) \label{eq:balance_for_u} \\[3mm]
\mu \dfrac{\partial u}{\partial x}\left( L,t\right) +\eta \dfrac{\partial
^{2}u}{\partial t\partial x}\left( L,t\right) =-P\left( t\right)
& & &
\mbox{for}\;0<t<T \label{eq:bc_x_L} \\
u\left( 0,t\right) =0 & & & \mbox{for}\;0<t<T \label{eq:bc_x_0} \\
u\left( x,0\right) =0 & & & \mbox{for}\;0<x<L \,.\label{eq:ic}
\end{align}%
\end{subequations}
Subsequently, we can recover the solid part of the stress tensor, discharge velocity,
pressure and total stress tensor on $(0,L)\times (0,T)$, respectively, as
\begin{subequations}
\label{eq:1d_problem}
\begin{eqnarray}
\sigma _{0} &=&\mu \dfrac{\partial u}{\partial x}+\eta \dfrac{\partial ^{2}u%
}{\partial t\partial x} \\
v &=&-\dfrac{\partial u}{\partial t}=-K\dfrac{\partial \sigma _{0}}{\partial
x}  \label{v} \\
p &=&\sigma _{0}+P\left( t\right)   \label{p} \\
\sigma  &=&-P\left( t\right) \,. \label{eq:1d_problem_end}
\end{eqnarray}
\end{subequations}

\begin{remark}
For later reference, we write below the explicit form of the purely elastic problem which corresponds to setting $\eta =0$ in \eqref{eq:1d_problem_per_u}:
\begin{subequations}
\label{eq:1d_problem_per_u_eta =0}
\begin{align}
\dfrac{\partial u}{\partial t}-K\mu \dfrac{\partial ^{2}u}{\partial x^{2}}%
=0 &  & & %
\mbox{in}\ (0,L)\times (0,T) \label{eq:balance_for_u_eta =0} \\[3mm]
\mu \dfrac{\partial u}{\partial x}\left( L,t\right) =-P\left( t\right)
& & &
\mbox{for}\;0<t<T \label{eq:bc_x_L_eta =0} \\
u\left( 0,t\right) =0 & & & \mbox{for}\;0<t<T \label{eq:bc_x_0_eta =0} \\
u\left( x,0\right) =0 & & & \mbox{for}\;0<x<L\,. \label{eq:ic_eta =0}
\end{align}%
\end{subequations}
\end{remark}

\begin{remark}
An important quantity associated with the fluid-solid mixture is the fluid
power density $\mathcal{P}\left( t\right)$, defined as
\begin{equation}
\mathcal{P}\left( t\right) =\int_{0}^{L}\frac{1}{K}\left\vert v\left(
x,t\right) \right\vert ^{2}dx.  \label{power}
\end{equation}
From its definition, it follows that $\mathcal{P}\left( t\right)$ depends on both discharge velocity and permeability.
\end{remark}

\subsection{Formal derivation of the solution}
\label{sec:elast}

Let us further assume that the permeability is constant, i.e.
\begin{equation*}
K=K\left(x,\frac{\partial u}{\partial x}\right)\equiv K_{0}=\mathrm{constant}%
>0.
\end{equation*}%
In this case, system (\ref{eq:1d_problem_per_u}) is a linear
initial boundary value problem,
whose solution can be obtained by Fourier series expansion as described below.

{\bf Case 1.} ($\eta >0$). First of all, we homogenize the boundary condition by
setting%
\begin{equation*}
w\left( x,t\right) =u\left( x,t\right) +\frac{U\left( t\right) }{\mu }x
\end{equation*}%
having introduced the auxiliary function
\begin{equation*}
U\left( t\right) =\frac{\mu }{\eta }\int_{0}^{t}\exp \left( -\frac{\mu }{%
\eta }\left( t-s\right) \right) P\left( s\right) ds=\frac{\mu }{\eta }\exp
\left( -\frac{\mu }{\eta }t\right) \ast P\left( t\right)
\end{equation*}%
where the star symbol denotes convolution. Thus, $w\left(
x,t\right) $ satisfies the problem%
\begin{equation}
\begin{array}{rll}
\dfrac{\partial w}{\partial t}-K_{0}\mu \dfrac{\partial ^{2}w}{\partial x^{2}%
}-K_{0}\eta \dfrac{\partial ^{3}w}{\partial t\partial x^{2}}=\dfrac{x}{\mu }%
U^{\prime }\left( t\right)  & \qquad  & \mbox{in}\ (0,L)\times (0,T) \\[3mm]
\mu \dfrac{\partial w}{\partial x}\left( L,t\right) +\eta \dfrac{\partial
^{2}w}{\partial t\partial x}\left( L,t\right) =0 & \qquad  & \mbox{for}%
\;0<t<T \\
w\left( 0,t\right) =0 & \qquad  & \mbox{for}\;0<t<T \\
w\left( x,0\right) =0 & \qquad  & \mbox{for}\;0<x<L%
\end{array}
\label{P_w}
\end{equation}%
where the prime symbol denotes differentiation for functions of a single variable.
The associated eigenvalue
problem is
\begin{center}
\it
Find $y=y\left( x\right) $, $0<x<L$, such that%
\begin{equation*}
y^{\prime \prime }+\lambda y=0,\qquad y\left( 0\right) =y^{\prime }\left(
L\right) =0\,.
\end{equation*}%
\end{center}
The eigenvalues and the corresponding eigenfunctions are given by
\begin{equation}
\lambda _{n}=\frac{(2n+1)^{2}\pi ^{2}}{4L^{2}}\quad \mbox{and}\quad
y_{n}(x)=\sin \frac{(2n+1)\pi x}{2L}  \label{lambda_y_n}\quad \mbox{for}\quad
n=0,1,\dots
\end{equation}%
We seek a
solution of the form%
\begin{equation*}
w\left( x,t\right) =\sum_{n=0}^{\infty }c_{n}\left( t\right) y_{n}\left(
x\right)
\end{equation*}%
where the coefficients $c_{n}\left( t\right) $ of the above series are determined
by the differential equation and the initial condition in (\ref{P_w}%
). Using the Fourier Sine series expansion%
\begin{equation*}
x=\frac{2}{L}\sum_{n=0}^{\infty }\frac{\left( -1\right) ^{n}}{\lambda _{n}}%
y_{n}\left( x\right) ,\qquad 0\leq x\leq L
\end{equation*}%
the uniqueness of the Fourier expansion leads to the family of ordinary differential equations%
\begin{equation*}
\left( 1+K_{0}\eta \lambda _{n}\right) c_{n}^{\prime }\left( t\right)
+K_{0}\mu \lambda _{n}c_{n}\left( t\right) =\frac{2\left( -1\right) ^{n}}{%
\mu L\lambda _{n}}U^{\prime }\left( t\right) ,\qquad c_{n}\left( 0\right) =0
\end{equation*}%
whose solution is given by%
\begin{equation*}
c_{n}\left( t\right) =\frac{2\left( -1\right) ^{n}}{\mu L\lambda _{n}\left(
1+K_{0}\eta \lambda _{n}\right) }\exp \left( -\frac{K_{0}\mu \lambda _{n}}{%
1+K_{0}\eta \lambda _{n}}t\right) \ast U^{\prime }\left( t\right)\,.
\end{equation*}%
Therefore, we get
\begin{equation*}
u\left( x,t\right) =-\frac{U\left( t\right) }{\mu }x+\frac{2}{\mu L}%
\sum_{n=0}^{\infty }\frac{\left( -1\right) ^{n}y_{n}\left( x\right) }{%
\lambda _{n}\left( 1+K_{0}\eta \lambda _{n}\right) }\exp \left( -\frac{%
K_{0}\mu \lambda _{n}}{1+K_{0}\eta \lambda _{n}}t\right) \ast U^{\prime
}\left( t\right) .
\end{equation*}
In conclusion, after performing integration by parts in the convolution term and using the
identity $\eta U^{\prime }\left( t\right) +\mu U\left( t\right) =\mu P\left(
t\right) $, the formal solution to problem~\eqref{eq:1d_problem_per_u} is given by
\begin{equation}
u\left( x,t\right) =-\frac{2K_{0}}{L}\sum_{n=0}^{\infty }\frac{\left(
-1\right) ^{n}y_{n}\left( x\right) }{1+K_{0}\eta \lambda _{n}}\exp \left( -%
\frac{K_{0}\mu \lambda _{n}}{1+K_{0}\eta \lambda _{n}}t\right) \ast P\left(
t\right) \,.  \label{solution 1}
\end{equation}

{\bf Case 2.} ($\eta =0$).
A direct substitution of expression~\eqref{solution 1} (with $\eta=0$)
into problem~\eqref{eq:1d_problem_per_u_eta =0} shows that
\begin{equation}
u\left( x,t\right) =-\frac{2K_{0}}{L}\sum_{n=0}^{\infty }
\left(-1\right) ^{n}y_{n}\left( x\right) \exp \left(-K_{0}\mu \lambda _{n} t\right)
\ast P\left( t\right)   \label{solution 2_bis}
\end{equation}
is the formal solution of the purely elastic problem.
In particular, in the case $P(t)=\mathrm{constant}$,
expression (\ref{solution 2_bis}) coincides with Eq.~(8) of~\cite{Soltz1998}.

\subsection{Weak solutions and well-posedness}
In this section we prove that the
formal solutions~\eqref{solution 1} and~\eqref{solution 2_bis} indeed solve the visco-elastic problem~\eqref{eq:1d_problem_per_u}  and the purely elastic problem~\eqref{eq:1d_problem_per_u_eta =0}, respectively, in well-defined functional spaces. Let us begin by introducing the
functional framework. We consider the real Hilbert
space $H=L^{2}\left( 0,L\right) $ equipped with the scalar product
\begin{equation*}
(f,g)_{H}=\int_{0}^{L}f(x)g(x)\,dx\qquad \forall f,g\in H,
\end{equation*}%
and endowed with the induced norm
\begin{equation*}
\left\Vert f\right\Vert _{H}=\sqrt{(f,f)_{H}}.
\end{equation*}%
The orthonormal sequence of eigenfunctions $\left\{ \sqrt{\frac{2}{L}}%
y_{n}\left( x\right) \right\} $ forms a Hilbert space basis for $H$. A
distribution $v$ in $\left( 0,L\right) $ belongs to $H$ if and only if it
has a series expansion (converging in $\mathcal{D}^{\prime }\left(
0,L\right) $)
\begin{equation}
v\left( x\right) =\sum_{n=0}^{\infty }c_{n}y_{n}\left( x\right),
\label{series}
\end{equation}%
with coefficients $c_{n}$ satisfying%
\begin{equation*}
\sum_{n=0}^{\infty }\left\vert c_{n}\right\vert ^{2}<\infty.
\end{equation*}%
If this is the case, the series expansion of $v$ converges in the mean.
Next, let
\begin{equation*}
V=\left\{ v\in H:v^{\prime }\in H,v\left( 0\right) =0\right\}
\end{equation*}%
be the real Hilbert space equipped with the scalar product
\begin{equation*}
\left( v,w\right) _{V}=\left( v^{\prime },w^{\prime }\right) _{H}\qquad \forall
v,w\in V,
\end{equation*}%
and endowed with the induced norm (due to Poincar{\'{e}}'s inequality)
\begin{equation*}
\left\Vert v\right\Vert _{V}=\left\Vert v^{\prime }\right\Vert _{H}.
\end{equation*}%
Sobolev's Embedding Theorem ensures that $V\subset C^{0}\left( \left[
0,L\right] \right) $, so that the boundary value $v\left( 0\right) =0$ for $%
v\in V$ is assumed in a strong sense. The sequence of functions $\left\{
\sqrt{\frac{2}{L\lambda _{n}}}y_{n}\left( x\right) \right\} $ forms a
Hilbert space basis for $V$. A function $v\in H$ belongs to $V$ if and only
if the Fourier coefficients of the series expansion (\ref{series}) satisfy%
\begin{equation*}
\sum_{n=0}^{\infty }\lambda _{n}\left\vert c_{n}\right\vert ^{2}<\infty.
\end{equation*}
More generally, the eigenfunction expansion (\ref{series}) enables us to define a
one-parameter family $V^{s}$ ($s\in \mathbb{R}$) of spaces: the
distribution (\ref{series}) belongs to $V^{s}$ if and only if its coefficients $c_{n}$
satisfy%
\begin{equation*}
\sum_{n=0}^{\infty }\lambda _{n}^{s}\left\vert c_{n}\right\vert ^{2} <\infty \,.
\end{equation*}%
In particular, we have that $V^{-1}\equiv V^{\prime }$ (the dual of $V$).

{\bf Case 1.} ($\eta >0$). 
Now we introduce the definition of weak solution for the poro-visco-elastic case.

\begin{definition}
\label{def weak sol}A function $u:\left[ 0,T\right] \rightarrow V$ is a weak
solution to problem (\ref{eq:1d_problem_per_u}) if:

\begin{description}
\item[(i)] $u\in H^{1}\left( 0,T;V\right) $, implying that $u,u^{\prime }\in L^{2}\left(
0,T;V\right) $;

\item[(ii)] for every $v\in V$ and for $t$ pointwise a.e. in $\left[ 0,T\right] $,
\begin{equation}
\left( u^{\prime }\left( t\right) ,v\right) _{H}+K_{0}\left( \mu u\left(
t\right) +\eta u^{\prime }\left( t\right) ,v\right) _{V}=-K_{0}P\left(
t\right) v\left( L\right) ;  \label{PDE weak}
\end{equation}

\item[(iii)] $u\left( 0\right) =0,$
\end{description}
where we used the notation $%
u\left( t\right) =u\left( \cdot ,t\right) $ and $u^{\prime }\left( t\right) =%
\dfrac{\partial u}{\partial t}\left( \cdot ,t\right) $.
\end{definition}

\begin{remark}
The initial condition (\ref{eq:ic}) is satisfied by $u\left( t\right)
\in V$ in the pointwise sense of condition (iii), since it is well known that
condition (i) implies that $u\in C^{0}\left( \left[ 0,T\right] ;V\right) $.
The Dirichlet boundary condition (\ref{eq:bc_x_0}) at $x=0$ is included in the requirement
that $u\left( t\right) \in V$, whereas the boundary condition (\ref{eq:bc_x_L}) at $x=L$ is
taken into account in the linear form on the right hand side of (\ref{PDE
weak}), where $v\left( L\right) $ is the pointwise value of the (continuous)
test function $v\in V$.
\end{remark}

\begin{theorem}
Suppose $P(t)\in L^{2}(0,T)$. Then there is a unique weak solution $u$ to problem (\ref{eq:1d_problem_per_u}), in the sense of Definition \ref{def weak sol}. Moreover, $%
u$ is represented by the series expansion (\ref{solution 1}).
\end{theorem}

\begin{proof}
For sake of exposition, we rewrite $u(x,t)$, given by (\ref{solution 1}), as
\begin{equation}
u\left( x,t\right) =\sum_{n=0}^{\infty }u_{n}\left( t\right) y_{n}\left(
x\right)   \label{u(t)}
\end{equation}%
where%
\begin{equation}
u_{n}\left( t\right) =-\frac{2K_{0}}{L}\frac{\left( -1\right) ^{n}}{%
1+K_{0}\eta \lambda _{n}}\exp \left( -\frac{K_{0}\mu \lambda _{n}}{%
1+K_{0}\eta \lambda _{n}}t\right) \ast P\left( t\right)   \label{u_n}
\end{equation}
so that%
\begin{equation}
\frac{\partial u}{\partial t}\left( x,t\right) =\sum_{n=0}^{\infty
}u_{n}^{\prime }\left( t\right) y_{n}\left( x\right)   \label{u'(t)}
\end{equation}
where
\begin{equation}
u_{n}^{\prime }\left( t\right) =-\frac{2K_{0}}{L}\frac{\left( -1\right) ^{n}%
}{1+K_{0}\eta \lambda _{n}}P\left( t\right) -\frac{K_{0}\mu \lambda _{n}}{%
1+K_{0}\eta \lambda _{n}}u_{n}\left( t\right) \label{u'_n}.
\end{equation}

\textbf{[Regularity].} Firstly, we show that $u(x,t)$ has the desired
regularity, namely $u\in H^{1}\left( 0,T;V\right) $ and $u\left( 0\right) =0$%
. With this goal in mind, we study the convergence of the series expansions
in (\ref{u(t)}) and (\ref{u'(t)}). For any $a>0$ Schwarz inequality implies that%
\begin{equation*}
\left\vert e^{-at}\ast P\left( t\right) \right\vert \leq \frac{1}{\sqrt{2a}}%
\left\Vert P\right\Vert _{L^{2}\left( 0,T\right) }.
\end{equation*}
Upon applying this estimate to (\ref{u_n}) and (\ref{u'_n}) we see that there are
positive constants, generically denoted by $C$,
that are independent on $n$ and $t$ and for which it holds that
\begin{equation}
\left\vert u_{n}\left( t\right) \right\vert \leq \frac{C}{\lambda _{n}}%
\left\Vert P\right\Vert _{L^{2}\left( 0,T\right) } \label{estimate u_n}
\end{equation}%
and%
\begin{equation}
\left\vert u_{n}^{\prime }\left( t\right) \right\vert \leq \frac{C}{\lambda
_{n}}\left( \left\vert P\left( t\right) \right\vert +\left\Vert P\right\Vert
_{L^{2}\left( 0,T\right) }\right) .\label{estimate u'_n}
\end{equation}%
Thus, $u,u^{\prime }\in L^{2}\left( 0,T;V^{s}\right) $  for every $s<3/2$,
thereby implying that  conditions (i) and (iii) of Def.~\ref{def weak sol} are satisfied.

\textbf{[Existence].} In order to show that $u$ satisfies condition (ii), it
suffices to consider $v=y_{n}$ as test function, since $\left\{ \sqrt{\frac{2%
}{L\lambda _{n}}}y_{n}\left( x\right) \right\} $ are a basis of $V$. Each
term on the left hand side of (\ref{PDE weak}) can be computed as
\begin{equation*}
\left( u^{\prime }\left( t\right) ,y_{n}\right) _{H}=\frac{L}{2}%
u_{n}^{\prime }\left( t\right)
\end{equation*}%
\begin{equation*}
\left( u\left( t\right) ,y_{n}\right) _{V}=\left( \dfrac{\partial u}{%
\partial x}\left( \cdot ,t\right) ,\dfrac{dy_{n}}{dx}\left( \cdot \right)
\right) _{H}=\frac{L}{2}\lambda _{n}u_{n}\left( t\right)
\end{equation*}%
\begin{equation*}
\left( u^{\prime }\left( t\right) ,y_{n}\right) _{V}=\left( \dfrac{\partial
^{2}u}{\partial t\partial x}\left( \cdot ,t\right) ,\dfrac{dy_{n}}{dx}\left(
\cdot \right) \right) _{H}=\frac{L}{2}\lambda _{n}u_{n}^{\prime }\left(
t\right) \,.
\end{equation*}%
Adding the above three terms, we obtain that%
\begin{equation*}
\left( u^{\prime }\left( t\right) ,y_{n}\right) _{H}+K_{0}\left( \mu u\left(
t\right) +\eta u^{\prime }\left( t\right) ,y_{n}\right) _{V}=\frac{L}{2}%
\left[ \left( 1+K_{0}\eta \lambda _{n}\right) u_{n}^{\prime }\left( t\right)
+K_{0}\mu \lambda _{n}u_{n}\left( t\right) \right]\,.
\end{equation*}%
Thus, from (\ref{u'_n}) and the fact that
$y_{n}\left( L\right) =\left( -1\right) ^{n}$, it follows that
the right hand side is equal to $-K_{0}P\left( t\right) y_{n}\left( L\right)
$, and this completes the proof of existence.

[\textbf{Uniqueness}] From the linearity of the equation, it suffices to
show that $P=0$ implies $u=0$. Let $P=0$ and let $u\in H^{1}\left(
0,T;V\right) $ be a solution of%
\begin{equation*}
\left( u^{\prime }\left( t\right) ,v\right) _{H}+K_{0}\left( \mu u\left(
t\right) +\eta u^{\prime }\left( t\right) ,v\right) _{V}=0,\qquad v\in V\,.
\end{equation*}%
If we choose $v=u\left( t\right) $ as a test function, then we obtain that
\begin{equation*}
\left( u^{\prime }\left( t\right) ,u\left( t\right) \right) _{H}+K_{0}\mu
\left\Vert u\left( t\right) \right\Vert _{V}^{2}+K_{0}\eta \left( u^{\prime
}\left( t\right) ,u\left( t\right) \right) _{V}=0
\end{equation*}%
implying that
\begin{equation*}
\frac{1}{2}\frac{d}{dt}\left( \left\Vert u\left( t\right) \right\Vert
_{H}^{2}+K_{0}\eta \left\Vert u\left( t\right) \right\Vert _{V}^{2}\right)
=-K_{0}\mu \left\Vert u\left( t\right) \right\Vert _{V}^{2}\leq 0\,.
\end{equation*}%
Integrating with respect to time and using the initial condition $u\left(
0\right) =0$, we get%
\begin{equation*}
\left\Vert u\left( t\right) \right\Vert _{H}^{2}+K_{0}\eta \left\Vert
u\left( t\right) \right\Vert _{V}^{2}\leq 0
\end{equation*}%
Hence $u\left( t\right) =0$ for every $t$ in $\left[ 0,T\right] $.
\end{proof}

\begin{remark}
From (\ref{estimate u_n}), the fact that $\left\vert y_{n}\left( x\right) \right\vert \leq 1$, and the
standard Weierstrass Test, it follows that the series (\ref{u(t)}) converges absolutely
and uniformly in the space-time rectangle $Q=\left[ 0,L\right] \times \left[ 0,T\right] $
and its sum is continuous, i.e. $u\left( x,t\right) \in C^{0}\left( Q\right)
$. In a similar way, under the stronger assumption $P\left( t\right) \in
L^{\infty }\left( 0,T\right) $, estimate (\ref{estimate u'_n}) implies $\left\vert
u_{n}^{\prime }\left( t\right) \right\vert \leq C\lambda _{n}^{-1}\left\Vert
P\right\Vert _{L^{\infty }\left( 0,T\right) }$ so that  $\frac{%
\partial u}{\partial t}\left( x,t\right) \in C^{0}\left( Q\right) $ also holds. Thus,
in the physically realistic case of bounded boundary traction $P\left(
t\right) $, discharge velocity and fluid power density are continuous and
bounded both in space and time, and this is true even though the
traction has jump discontinuities (e.g., it is a step pulse).
\end{remark}

{\bf Case 2.} ($\eta =0$). In the purely elastic problem, the formal solution (\ref{solution 2_bis}) is again
written as given in (\ref{u(t)}) where now%
\begin{equation*}
u_{n}\left( t\right) =-\frac{2K_{0}}{L}\left( -1\right) ^{n}\exp \left(
-K_{0}\mu \lambda _{n}t\right) \ast P\left( t\right).
\end{equation*}%
Then, estimates (\ref{estimate u_n}) and (\ref{estimate u'_n}) become, respectively,%
\begin{equation*}
\left\vert u_{n}\left( t\right) \right\vert \leq \frac{C}{\sqrt{\lambda _{n}}%
}\left\Vert P\right\Vert _{L^{2}\left( 0,T\right) }
\end{equation*}%
and%
\begin{equation*}
\left\vert u_{n}^{\prime }\left( t\right) \right\vert \leq C\sqrt{\lambda
_{n}}\left( \left\vert P\left( t\right) \right\vert +\left\Vert P\right\Vert
_{L^{2}\left( 0,T\right) }\right).
\end{equation*}%
As a consequence, it can only be asserted that $u\in L^{2}\left( 0,T;V^{s}\right) $
for every $s<1/2$ and $u^{\prime }\in L^{2}\left( 0,T;V^{s}\right) $ for
every $s<-3/2$. On the other hand, if one assumes $P\left( t\right) \in
L^{\infty }\left( 0,T\right) $, we find that%
\begin{equation*}
\left\vert u_{n}\left( t\right) \right\vert \leq \frac{C}{\lambda _{n}}%
\left\Vert P\right\Vert _{L^{\infty }\left( 0,T\right) }
\end{equation*}%
and%
\begin{equation*}
\left\vert u_{n}^{\prime }\left( t\right) \right\vert \leq C\left\Vert
P\right\Vert _{L^{\infty }\left( 0,T\right) }.
\end{equation*}%
Hence, in this case, $u\in L^{2}\left( 0,T;V^{s}\right) $ for every $s<3/2$
and $u^{\prime }\in L^{2}\left( 0,T;V^{s}\right) $ for every $s<-1/2$. In
particular, when the boundary traction is bounded, we have $u\left(
x,t\right) \in C^{0}\left( Q\right) $ whereas the discharge velocity is a
distribution. In conclusion, we state the following definition and theorem (without proof):

\begin{definition}
\label{def weak sol eta=0}A function $u:\left[ 0,T\right] \rightarrow V$ is a weak
solution to problem~\eqref{eq:1d_problem_per_u_eta =0} if:

\begin{description}
\item[(i)] $u\in L^{2}\left( 0,T;V\right) $, $u^{\prime }\in L^{2}\left(
0,T;V^{\prime }\right) $;

\item[(ii)] for every $v\in V$ and for $t$ pointwise a.e. in $\left[ 0,T%
\right] $,
\begin{equation}
\left\langle u^{\prime }\left( t\right) ,v\right\rangle +K_{0}\mu \left(
u\left( t\right) ,v\right) _{V}=-K_{0}P\left( t\right) v\left( L\right)
\label{PDE weak eta=0}
\end{equation}
where the brackets $\left\langle \cdot ,\cdot \right\rangle $ denote pairing between $V^{\prime }$ and $V$;

\item[(iii)] $u\left( 0\right) =0.$
\end{description}
\end{definition}

\begin{theorem}
Suppose $P(t)\in L^{\infty }(0,T)$. Then there exists a unique weak solution $u$
to problem~\eqref{eq:1d_problem_per_u_eta =0}, in the sense of Definition \ref{def weak sol eta=0}. Moreover, $u$
is represented by the series expansion~\eqref{solution 2_bis}.
\end{theorem}

\section{Dimensionless Problem}\label{sec:dimensionless_problem}

The goal of this section is to rewrite problem \eqref{eq:1d_problem_per_u} in dimensionless form so that we can identify combinations of geometrical and physical parameters that most influence the solution properties. Dimensional analysis relies on the choice of a set of characteristic values that can be used to scale all the problem variables. Let us use the hat symbol to indicate dimensionless (or scaled) variables and the square brackets to indicate the characteristic value of that quantity. Then, for the problem at hand we would write:
\begin{equation} \label{starred variables}
\begin{array}{c}
\displaystyle\hx =\frac{x}{[x]},\qquad \hatt=\frac{t}{[t]},\qquad
\heta=\frac{\eta}{[\eta]}, \qquad \hlambda _{n} =\frac{\lambda _{n}}{[\lambda _{n}]},  \\[.1in]
\displaystyle\hat{P} = \frac{P}{[P]}, \qquad   \hu =\frac{u}{[u]},\qquad \hv=\frac{v}{[v]}, \qquad
 \hat{\mathcal{P}} =\frac{\mathcal P}{[\mathcal{P}]} \,.
\end{array}
\end{equation}
It is important to emphasize that there is no trivial choice for the characteristic values and, in general, this choice is not unique. In this particular case, though, we will leverage our knowledge of the forcing terms and the explicit formulas we obtained for the solution to guide us in the choice of some of these values. Since the problem is driven by the boundary condition on the traction with the given function $P(t)$, we set
\begin{equation}
[P]=P_{\text{ref}}
\end{equation}
where $P_{\text{ref}} $ is a reference value, for example the mean value, of the given function $P(t)$.
The expression for $\lambda_n$ in \eqref{lambda_y_n} suggests to choose
\begin{equation}
[\lambda_n]=\frac{1}{L^2} \quad \mbox{and} \quad [x]=L
\end{equation}
and, consequently, the expression for $u(x,t)$ in~\eqref{solution 1} suggests
to choose
\begin{equation}\label{eq:scalf_eta}
[\eta] = \frac{1}{K_0[\lambda_n]} = \frac{L^2}{K_0}, \quad
[t] =  \frac{1}{K_0\mu[\lambda_n]} = \frac{L^2}{K_0\mu},
\quad
[u] = \frac{K_0}{L}[P][t] = \frac{P_{\text{ref}} L}{\mu}
\,,
\end{equation}
whereas the expressions for $v(x,t)$ in~\eqref{v} and $\mathcal P(t)$ in~\eqref{power} suggest to choose
\begin{equation}\label{eq:scalf_velocity}
[v] = \frac{[u]}{[t]} =  \frac{K_0}{L}P_{\text{ref}}
\end{equation}
and
\begin{equation}\label{eq:scalf_power_density}
[\mathcal P] = \frac{L}{K_0}[v]^2 =  \frac{K_0}{L}P_{\text{ref}}^2
\end{equation}
respectively. Using the above scalings, we obtain the following dimensionless problem:
\begin{subequations}
\label{eq:1d_problem_per_hu}
\begin{align}
\dfrac{\partial \hu}{\partial \hatt}- \dfrac{\partial ^{2} \hu}{\partial \hx ^{2}}%
-\heta \dfrac{\partial ^{3}\hu}{\partial \hatt\partial \hx ^{2}}=0 &  & & %
\mbox{in}\ (0,1)\times (0,\hT) \label{eq:balance_for_hu} \\[3mm]
 \dfrac{\partial \hu}{\partial \hx}\left( 1,\hatt\right) +\heta \dfrac{\partial
^{2}\hu}{\partial \hatt\partial \hx}\left( 1, \hatt\right) =-\hP\left( \hatt\right)
& & &
\mbox{for}\;0<\hatt<\hT \label{eq:bc_hx_1} \\
\hu\left( 0,\hatt\right) =0 & & & \mbox{for}\;0<\hatt<\hT \label{eq:bc_hx_0} \\
\hu\left( \hx,0\right) =0 & & & \mbox{for}\;0<\hx<1 \label{eq:ic h}
\end{align}%
\end{subequations}
where $\hT = T/[t]$ and $\heta \geq 0$. Recalling again the expressions~\eqref{solution 1},~\eqref{v} and~\eqref{power}, we can now write solid displacement, discharge velocity and power density in dimensionless form as
\begin{equation}
\hu\left( \hx,\hatt\right) =-2\sum_{n=0}^{\infty }\frac{\left(
-1\right) ^{n}y_{n}\left( \hx\right) }{1+\heta \hlambda _{n}}\exp \left( -%
\frac{ \hlambda _{n}}{1+\heta \lambda _{n}}\hatt\right) \ast \hP\left(
\hatt\right)   \label{solution 1hat}
\end{equation}%
\begin{equation} \label{hv}
\hv\left( \hx,\hatt\,\right) =2\sum_{n=0}^{\infty }\frac{\left(
-1\right) ^{n}y_{n}\left( \hx\right) }{1+\heta \hlambda _{n}} \left\{ \hat{P}\left( \hatt\,\right) -\frac{%
\hlambda _{n}}{1+\heta \hlambda _{n}}\exp \left( -%
\frac{ \hlambda _{n}}{1+\heta \lambda _{n}}\hatt\right) \ast \hP\left(
\hatt\right)\right\}
\end{equation}%
\begin{equation}\label{power hat}
\hat{\mathcal{P}}\left( \hatt\right) =2\sum_{n=0}^{\infty }\frac{1}{\left(
1+\heta \hlambda _{n}\right) ^{2}}\left\{ \hat{P}\left( \hatt\,\right) -\frac{\hlambda _{n}%
}{1+\heta \hlambda _{n}}\exp \left( -%
\frac{ \hlambda _{n}}{1+\heta \lambda _{n}}\hatt\right) \ast \hP\left(
\hatt\right) \right\} ^{2}\,.
\end{equation}

\begin{remark}
As already mentioned above, the choice for the characteristic values is not unique. In this regard, it is worth noticing that our choice for $[v]$, namely $K_0 P_{\text{ref}}/L$, differs from that utilized in~\cite{Soltz1998}, which was $K_0 \mu/L$ instead. This difference actually follows from a different choice for the pressure scaling, namely $P_{\text{ref}}$ in our case, as opposed to $\mu$ in~\cite{Soltz1998}. Indeed, the two scalings are equivalent if $P_{\text{ref}}$  and $\mu$ are of the same order of magnitude. In general, though, the scaling $K_0 \mu/L$ represents the characteristic velocity of wave propagation within the medium, whereas the scaling $K_0 P_{\text{ref}}/L$ represents the characteristic velocity induced by an external load of magnitude $P_{\text{ref}}$. Since our analysis aims at assessing the influence of the external load on the biomechanical response of the medium, we chose $[P]=P_{\text{ref}}$. However, the analysis could be easily adapted to other choices, should the scope of the investigation differ from that of this paper.
\end{remark}

\section{The case $\hat{P}\left( \hatt\,\right) =\mathrm{step}$ $\mathrm{pulse}$}
\label{sec:step_pulse}
%
%
Let the boundary traction $\hat{P}\left( \hatt\,\right) $ be the dimensionless unit step starting at
$\hatt=0$ defined as
\begin{equation}\label{eq:H}
\hat{P}\left( \hat{t}\,\right) =H\left( \hat{t}\,\right) =\left\{
\begin{array}{lll}
0 & \text{if} & \hat{t}<0 \\
1 & \text{if} & \hat{t}\geq 0%
\end{array}%
\right. \,.
\end{equation}%
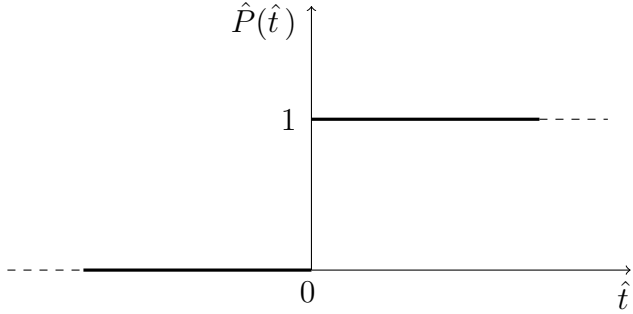
\begin{figure}[h!]
\begin{center}
\begin{tikzpicture}
\node[below] at (-0.05,0) {$0$};
\draw[dashed] (-4,0) -- (-3,0);
\draw[->] (0,0) -- (4.2,0);
\node[below] at (4.1,0) {$\hatt$};
\draw[->] (0,0) -- (0,3.5);
\node[left] at (-0.05,3.3) {$\hat{P}(\hat{t}\,)$};
\draw[very thick] (-3,0) -- (0,0);
\draw[very thick] (0,2) -- (3,2);
\draw[dashed] (3,2) -- (3.9,2);
\node[left] at (-0.05,2) {$1$};
\end{tikzpicture}
\end{center}
\caption{Schematic representation of the dimensionless step pulse $\hat{P}(\hatt\,)$ defined in~\eqref{eq:H}. Here, the signal is discontinuous at the switch on time.}
\label{fig:step_pulse}
\end{figure}
In this case, the dimensionless solid displacement \eqref{solution 1hat}, discharge velocity \eqref{hv} and
power density \eqref{power hat} (hereon denoted with the subscript $\heta$ for later reference) become,
respectively,
\begin{subequations}
\label{eq:solution_hat_step_pulse}
\begin{align}
\hu_{\heta}\left( \hx,\hatt\,\right) & =-2\sum_{n=0}^{\infty }\frac{\left(
-1\right) ^{n}}{\hlambda_{n}}\left\{ 1-\exp \left( -\frac{\hlambda_{n}\hatt}{%
1+\heta\hlambda_{n}}\right) \right\} y_{n}\left( \hx\right)  & &  \notag \\
& =-\hx+2\sum_{n=0}^{\infty }\frac{\left( -1\right) ^{n}}{\hlambda_{n}}\exp
\left( -\frac{\hlambda_{n}\hatt}{1+\heta\hlambda_{n}}\right) y_{n}\left( \hx%
\right)  & &  \label{u hat step} \\
\hv_{\heta}\left( \hx,\hatt\right) & =2\sum_{n=0}^{\infty }\frac{\left(
-1\right) ^{n}}{1+\heta\hlambda_{n}}\exp \left( -\frac{\hlambda_{n}\hatt}{1+%
\heta\hlambda_{n}}\right) y_{n}\left( \hx\right)  & & \\
\hat{\mathcal{P}}_{\heta}\left( \hatt\right) & =2\sum_{n=0}^{\infty }\frac{1%
}{\left( 1+\heta\hlambda_{n}\right) ^{2}}\exp \left( -\frac{2\hlambda_{n}%
\hatt}{1+\heta\hlambda_{n}}\right) \,. & &  \label{power hat step}
\end{align}
\end{subequations}
The solution in the purely elastic case can be obtained by setting
$\heta =0$ in~\eqref{eq:solution_hat_step_pulse}.
\begin{remark}
If the unit step is shifted at $\hat{t}=\alpha >0$, i.e. $\hat{P}\left( \hat{%
t}\right) =H\left( \hat{t}-\alpha \right) $, then the solution \eqref{solution 1hat} is%
\begin{equation*}
\hat{u}\left( \hat{x},\hat{t}\right) =\hat{u}_{\hat{\eta}}\left( \hat{x},%
\hat{t}-\alpha \right) H\left( \hat{t}-\alpha \right) =\left\{
\begin{array}{lll}
0 & \text{if} & 0\leq \hat{t}<\alpha  \\
\hat{u}_{\hat{\eta}}\left( \hat{x},\hat{t}-\alpha \right)  & \text{if} &
\hat{t}\geq \alpha.
\end{array}%
\right.
\end{equation*}
\end{remark}
The space-time behavior of $\hu_{\heta}$, $\hv_{\heta}$ and $\hat{\mathcal P}_{\heta}$ is reported in Figures~\ref{fig:displacement_step_pulse},~\ref{fig:discharge_velocity_step_pulse}
and~\ref{fig:power_density_step_pulse}, respectively, in the case of $\heta=0$ and $\heta=1$.
We remark that $\hv_{\heta}$ and
$\hat{\mathcal P}_{\heta}$ are reported in logarithmic scale to highlight
the presence of a blow-up at $\hat{x}=1$ and $\hatt=0$ in the purely elastic case $\heta=0$.
Indeed, by setting
$\heta=\hatt=0$ in~\eqref{power hat step}, we can verify that the power density satisfies
$$
\hat{\mathcal{P}}_{0}\left( 0 \right) =2\sum_{n=0}^{\infty } 1 = +\infty.
$$
Proceeding analogously in the case of the discharge velocity, the Fourier expansion
in the purely elastic case is given by
\begin{equation*}
\hat{v}_{0}\left( \hat{x},0\right) =2\sum_{n=0}^{\infty }\left( -1\right)
^{n}y_{n}\left( \hat{x}\right)
\end{equation*}%
which clearly lacks pointwise convergence for any $\hat{x}\neq 0$.

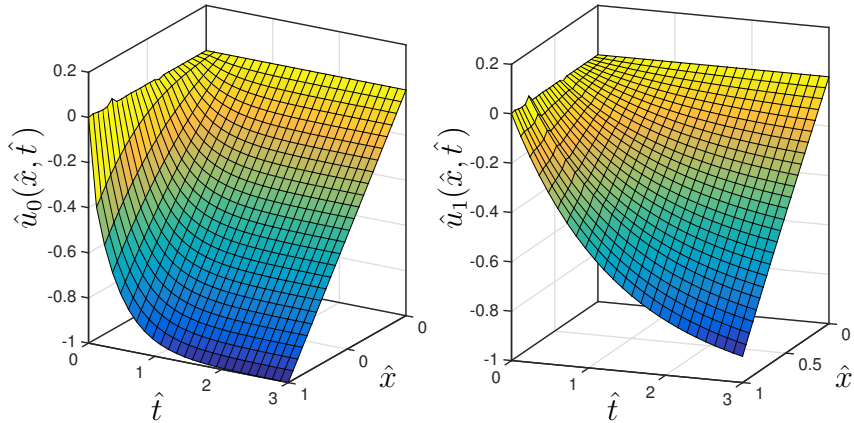
\begin{figure}[h!]
\begin{center}
\begin{tikzpicture}
\node at (0,0)
{\includegraphics[scale=0.35]{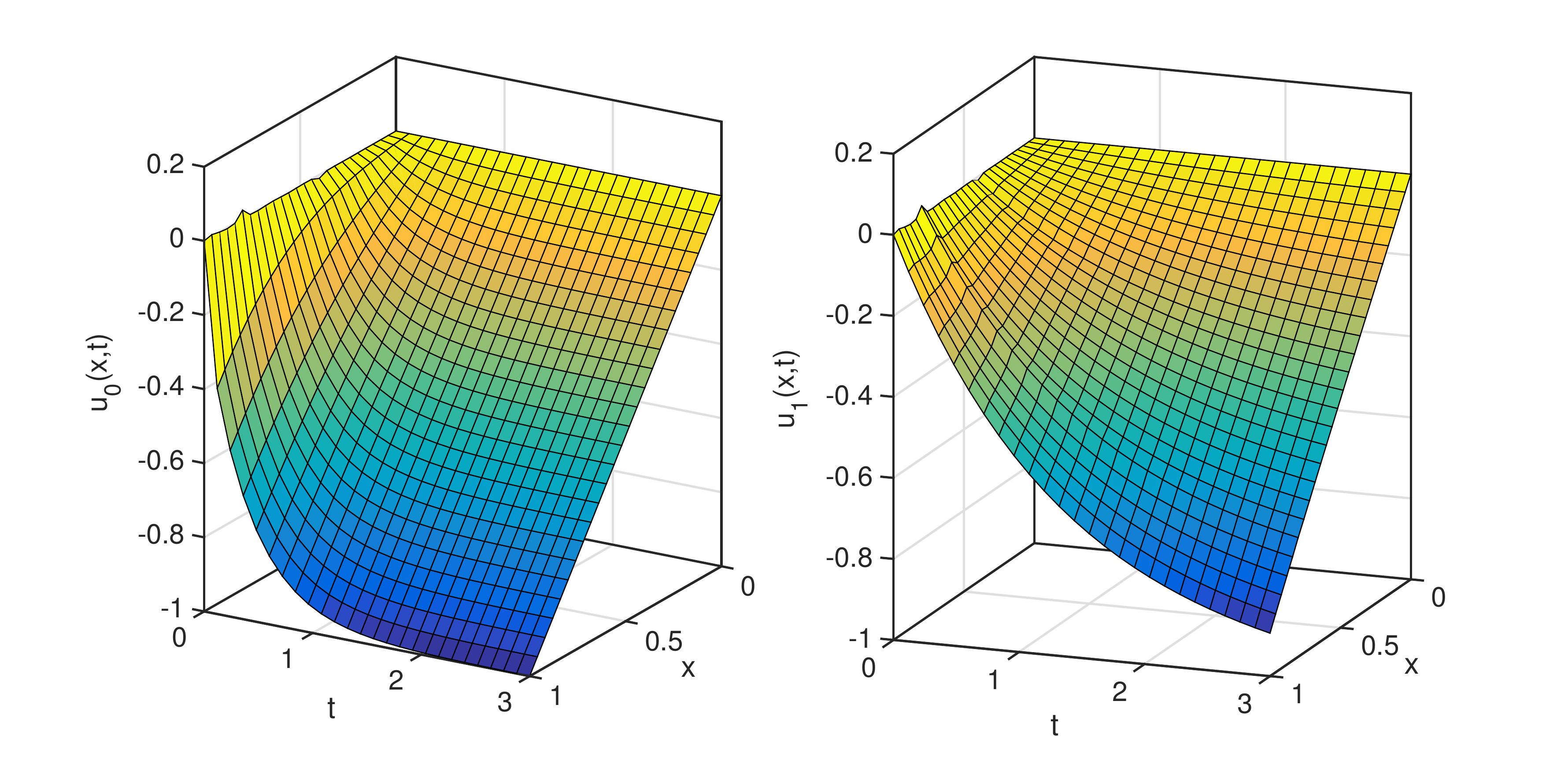}};
\fill [white] (-1,-2.5) rectangle (-0.5,-2);
\fill [white] (5,-2.5) rectangle (5.5,-2);
\node [right] at (-1,-2.3) {$\hat{x}$};
\node [right] at (5,-2.3) {$\hat{x}$};
\fill [white] (-4,-3) rectangle (-3.5,-2.5);
\fill [white] (2,-3) rectangle (2.5,-2.5);
\node [right] at (-4,-2.7) {$\hat{t}$};
\node [right] at (2,-2.7) {$\hat{t}$};
\fill [white] (-5.8,-0.5) rectangle (-5.3,1);
\node [right,rotate=90] at (-5.5,-0.5) {$\hat{u}_0(\hat{x},\hat{t}\,)$};
\fill [white] (-0.2,-0.5) rectangle (0.3,1);
\node [right,rotate=90] at (0.1,-0.5) {$\hat{u}_1(\hat{x},\hat{t}\,)$};
\end{tikzpicture}
\end{center}
\caption{Dimensionless displacement $\hu_{\heta}$ as a function of $\hx$ and $\hatt$.
Left panel: $\heta=0$. Right panel: $\heta=1$.}
\label{fig:displacement_step_pulse}
\end{figure}

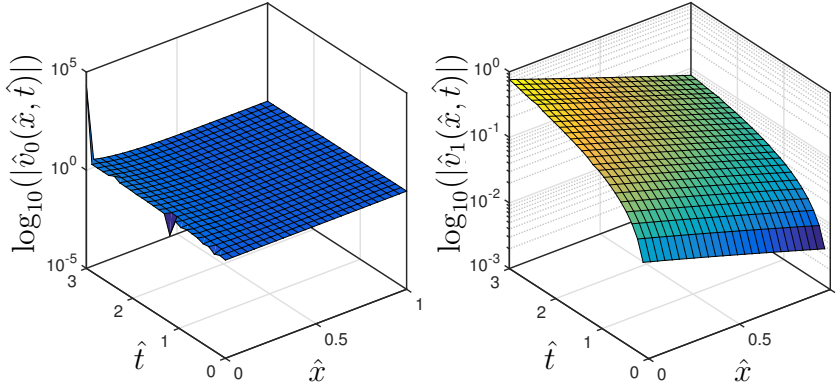
\begin{figure}[h!]
\begin{center}
\begin{tikzpicture}
\node at (0,0)
{\includegraphics[scale=0.35]{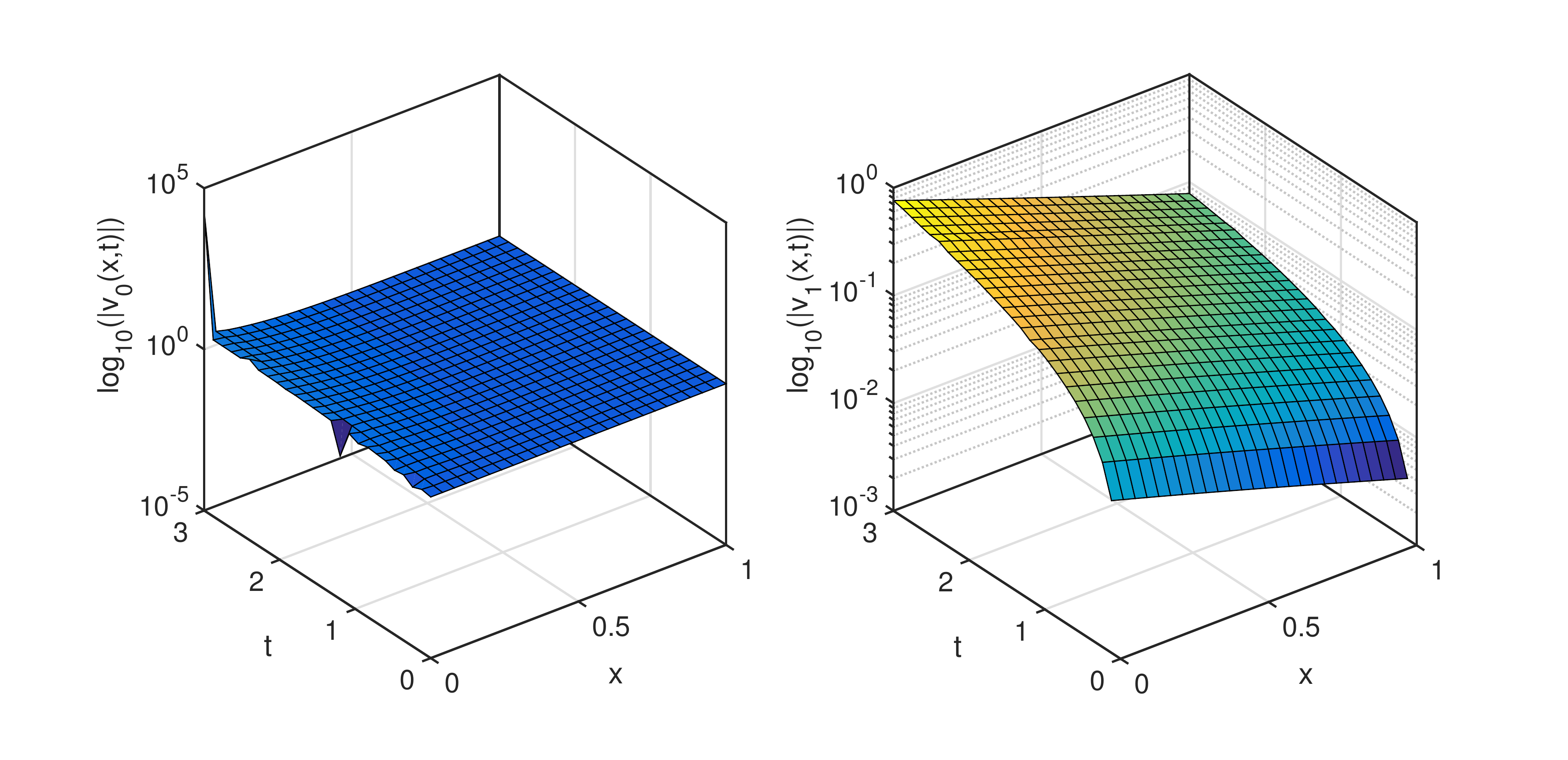}};
\fill [white] (-2,-2.7) rectangle (-1,-2.2);
\node [right] at (-1.9,-2.4) {$\hat{x}$};
\fill [white] (4,-2.7) rectangle (5,-2.2);
\node [right] at (3.7,-2.4) {$\hat{x}$};
\fill [white] (-4,-2.5) rectangle (-4.5,-2);
\node [right] at (-4.2,-2.2) {$\hat{t}$};
\fill [white] (1,-2.5) rectangle (1.5,-2);
\node [right] at (1.2,-2.2) {$\hat{t}$};
\fill [white] (-5.7,-1) rectangle (-5.2,2);
\node [right,rotate=90] at (-5.5,-1) {$\log_{10}(|\hat{v}_0(\hat{x},\hat{t})|)$};
\fill [white] ( -0.2,-1) rectangle (0.3,2);
\node [right,rotate=90] at (0.1,-1) {$\log_{10}(|\hat{v}_1(\hat{x},\hat{t})|)$};
\end{tikzpicture}
\end{center}
\caption{Dimensionless discharge velocity $\hv_{\heta}$ as a function of $\hx$ and $\hatt$.
Left panel: $\heta=0$. Right panel: $\heta=1$. In order to highlight the velocity blow-up at
$\hx=1$, $\hatt=0$, the $\log_{10}$ plot of $|\hv_{\heta}|$ is plotted in both panels. }
\label{fig:discharge_velocity_step_pulse}
\end{figure}

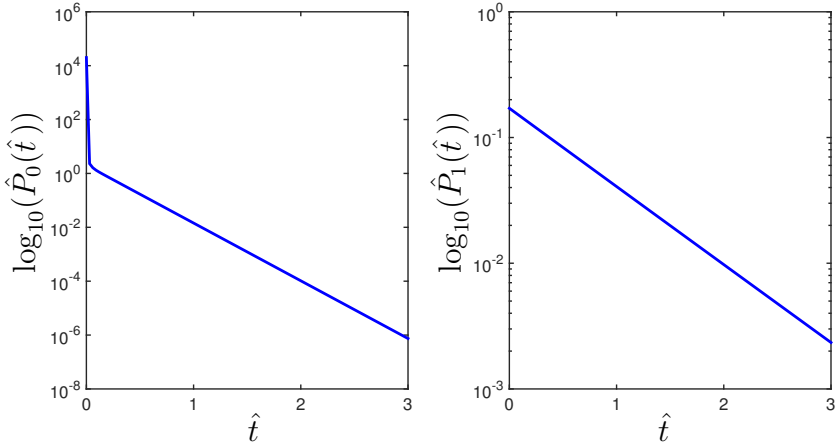
\begin{figure}[h!]
\begin{center}
\begin{tikzpicture}
\node at (0,0)
{\includegraphics[scale=0.35]{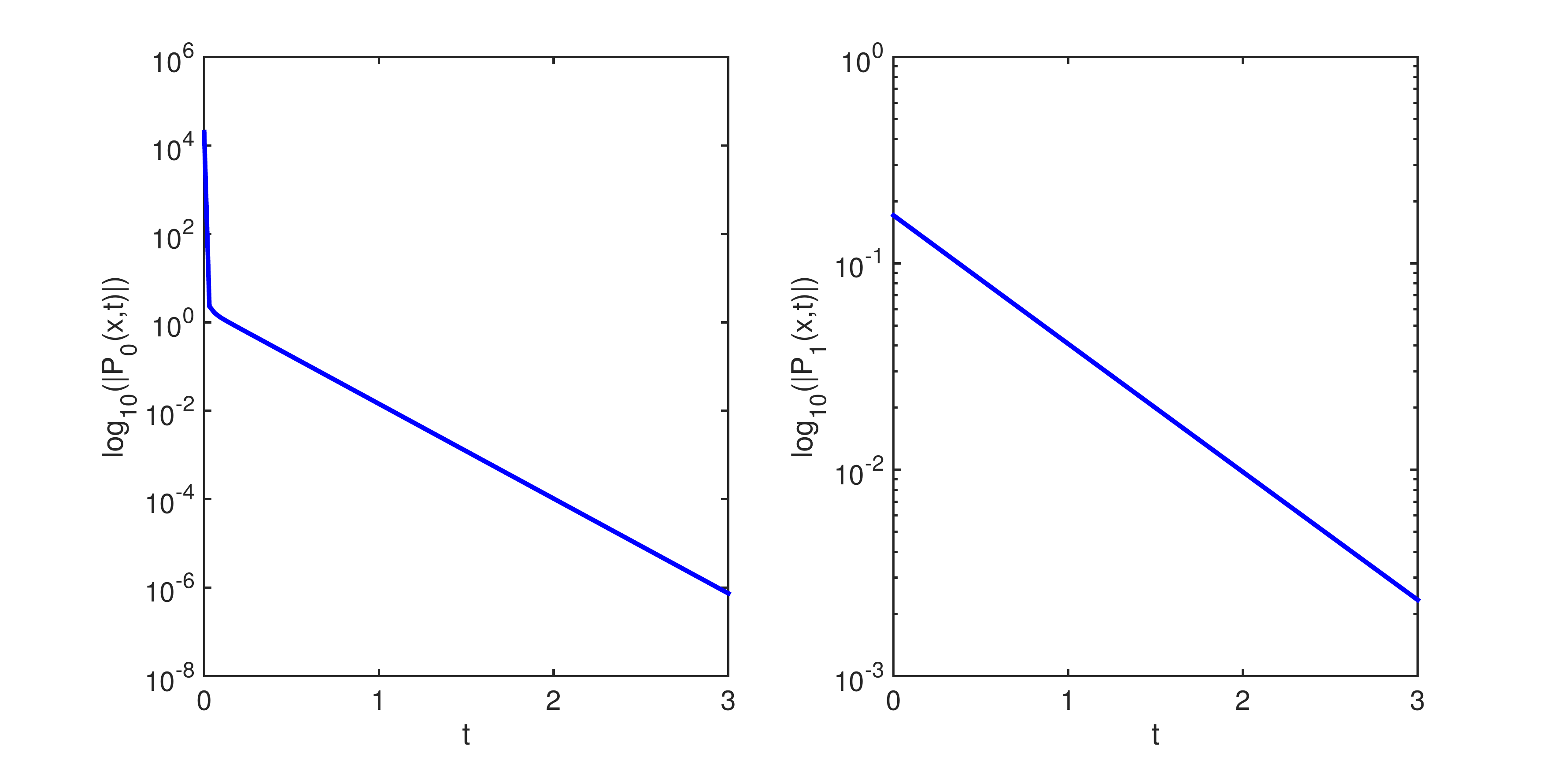}};
\fill [white] (-3.3,-3.1) rectangle (-2.2,-2.7);
\node [right] at (-2.7,-2.9) {$\hat{t}$};
\fill [white] (2.2,-3.1) rectangle (3.3,-2.7);
\node [right] at (2.7,-2.9) {$\hat{t}$};
\fill [white] (-5.7,-1) rectangle (-5.2,1);
\node [right,rotate=90] at (-5.5,-1) {$\log_{10}(\hat{P}_0(\hat{t}\,))$};
\fill [white] (-0.2,-1) rectangle (0.35,1);
\node [right,rotate=90] at (0.1,-1) {$\log_{10}(\hat{P}_1(\hat{t}\,))$};
\end{tikzpicture}
\end{center}
\caption{Dimensionless power density $\hat{\mathcal{P}}_{\heta}$ as a function of $\hatt$. Left panel: $\heta=0$. Right panel: $\heta=1$. In order to highlight the power density blow-up at $\hatt=0$, the $\log_{10}$ plot of $\hat{\mathcal{P}}_{\heta}$ is plotted in both panels.}
\label{fig:power_density_step_pulse}
\end{figure}

From the physical viewpoint, this means that,
at the switch on time of the driving term, here set at $\hatt=0$, the discharge velocity contains high-frequency
components whose superposition exhibits the following behaviors depending on $\hat{x}$:
(i) it relaxes
towards zero as $\hat{x}\rightarrow 0$; (ii) it tends to infinity as $\hat{x}%
\rightarrow 1$, thereby exhibiting a blow-up; (iii) it looks like an
amorphous blob for $0<\hat{x}<1$.
On the other hand, if viscoelastic effects are present, then the dimensionless discharge velocity at $\hatt=0$ is given by
\begin{equation*}
\hat{v}_{\hat{\eta}}\left( \hat{x},0\right) =2\sum_{n=0}^{\infty }\frac{%
\left( -1\right) ^{n}}{1+\hat{\eta}\hat{\lambda}_{n}}y_{n}\left( \hat{x}%
\right)\,.
\end{equation*}%
Here, the $n-$th coefficient is decreasing as $1/n^{2}$ and the discharge
velocity is continuous at $\hat{t}=0$. These concepts are illustrated in Fig.~\ref{fig:discharge_velocity_t_0}, which shows the spatial distribution of $\hat{v}_{\hat{\eta}}$ at $\hatt=0$ for several values of $\hat{\eta}$
in the interval $[0,1]$. The $\log_{10}$ plot of $\hat{v}_{\hat{\eta}}$ is reported in the figure to highlight
the blow-up of the discharge velocity in the purely elastic case at $\hat{x}=1$.
\begin{figure}[h!]
\begin{center}
\begin{center}
\begin{tikzpicture}
\node at (0,0)
{\includegraphics[scale=0.35]{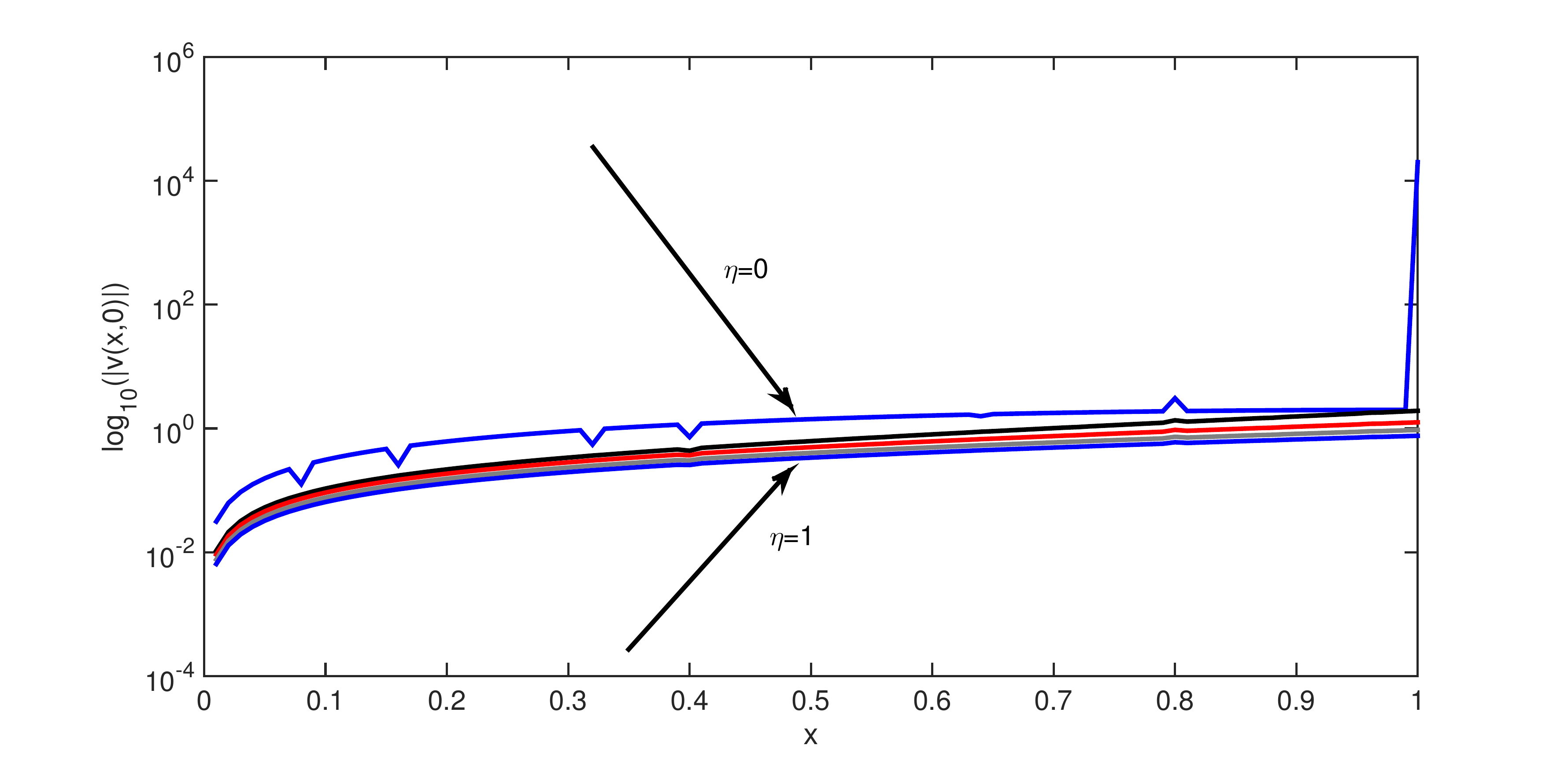}};
\fill [white] (-1,-3.1) rectangle (1,-2.7);
\node [right] at (-0.5,-2.9) {$\hat{x}$};
\fill [white] (-5.7,-1) rectangle (-5.2,1.5);
\node [right,rotate=90] at (-5.5,-1) {$\log_{10} |\hat{v}_{\hat{\eta}}(\hat{x},0)|$};
\end{tikzpicture}
\end{center}
\caption{Dimensionless discharge velocity $\hv_{\heta}(\hx,0)$ as a function of
$\hx$. In order to highlight the velocity blow-up at $\hx=1$, $\hatt=0$, in the
case $\heta=0$, the $\log_{10}$ plot of $|\hv_{\eta}|$ is plotted.}
\label{fig:discharge_velocity_t_0}
\end{center}
\end{figure}

In order to further investigate this blow-up and its dependence on the structural viscoelasticity, we observe that the maximum value of $\hv_{\heta}$ is attained at $\hx=1$ and $\hatt=0$ and can be written as a function of $\heta$ as follows:
\begin{equation*}
\hat{v}_{\max }\left( \hat{\eta}\right) =\max_{\substack{ 0\leq \hx\leq 1 \\ \hatt \geq 0}}\left\vert
\hv_{\heta }\left( \hx,\hatt\right) \right\vert =\hv_{\heta }\left( 1,0\right)
=2\sum_{n=0}^{\infty }\frac{1}{1+\heta \hlambda _{n}}\,.
\end{equation*}
The above series may be summed (\cite{Gradshteyn-Ryzhik}, formula no. 1.4212) and the final result is
\begin{equation}
\hat{v}_{\max }\left( \hat{\eta}\right) =\frac{1}{\sqrt{\hat{\eta}}}\tanh \left( \frac{1}{%
\sqrt{\hat{\eta}}}\right) \label{M(eta)}\,.
\end{equation}
Similarly, the dimensionless power density \eqref{power hat step} is decreasing in time and its
maximum is attained at $\hat{t}=0$:%
\begin{eqnarray*}
\hat{\mathcal{P}}_{\max }\left( \hat{\eta}\right)  &=&\max_{\hat{t}\geq 0}%
\hat{\mathcal{P}}_{\hat{\eta}}\left( \hat{t}\right) =\hat{\mathcal{P}}_{\hat{%
\eta}}\left( 0\right) =2\sum_{n=0}^{\infty }\frac{1}{\left( 1+\hat{\eta}\hat{%
\lambda}_{n}\right) ^{2}} \\
&=&\frac{1}{2\hat{\eta}}\left( \tanh ^{2}\frac{1}{\sqrt{\hat{\eta}}}+\sqrt{%
\hat{\eta}}\tanh \frac{1}{\sqrt{\hat{\eta}}}-1\right)\,.
\end{eqnarray*}%
The behaviors of $\hat{v}_{\max }\left( \hat{\eta}\right) $ and $\hat{\mathcal{P}}_{\max }\left( \hat{\eta}\right) $ with respect to $\heta$ are visualized in Fig.~\ref{fig:M_eta_step_pulse}.
\begin{figure}[h!]
\begin{center}
\begin{tikzpicture}
\node at (0,0)
{\includegraphics[scale=0.35]{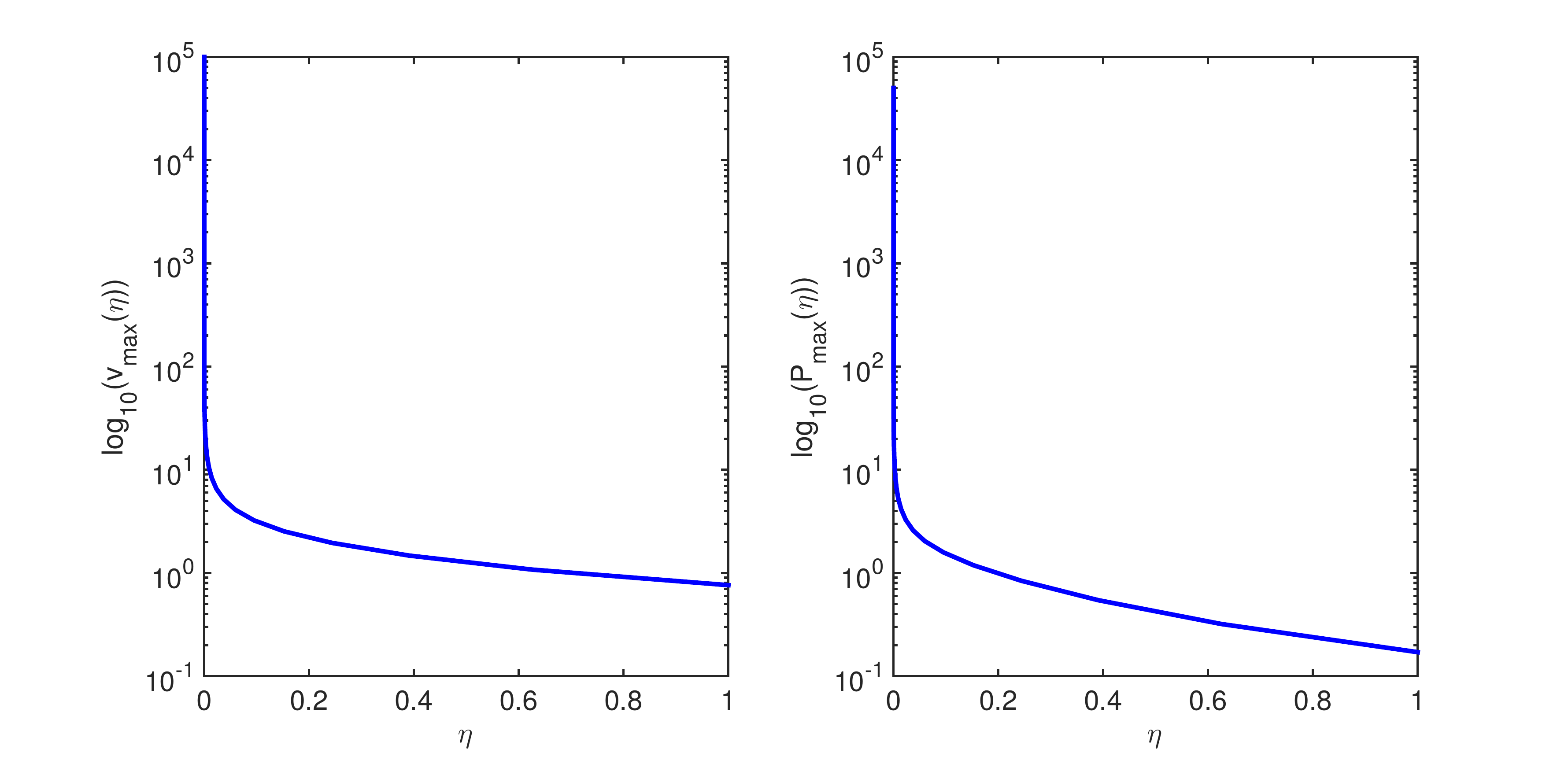}};
\fill [white] (-3,-3.1) rectangle (-2,-2.7);
\node [right] at (-2.8,-2.9) {$\hat{\eta}$};
\fill [white] (2.5,-3.1) rectangle (3.5,-2.7);
\node [right] at (2.7,-2.9) {$\hat{\eta}$};
\fill [white] (-5.7,-1) rectangle (-5.2,1);
\node [right,rotate=90] at (-5.5,-1) {$\log_{10}(\hat{v}_{\max}(\hat{\eta}))$};
\fill [white] (-0.1,-1) rectangle (0.4,1);
\node [right,rotate=90] at (0.1,-1) {$\log_{10}(\hat{P}_{\max}(\hat{\eta}))$};
\end{tikzpicture}
\caption{Left panel: maximal discharge velocity as a function of $\heta$.
Right panel: power density as a function of $\heta$. Log$_{10}$-scale
is used on the $y$-axis to
better highlight blow-up of both quantities as $\hat{\eta} \rightarrow 0$.}
\label{fig:M_eta_step_pulse}
\end{center}
\end{figure}
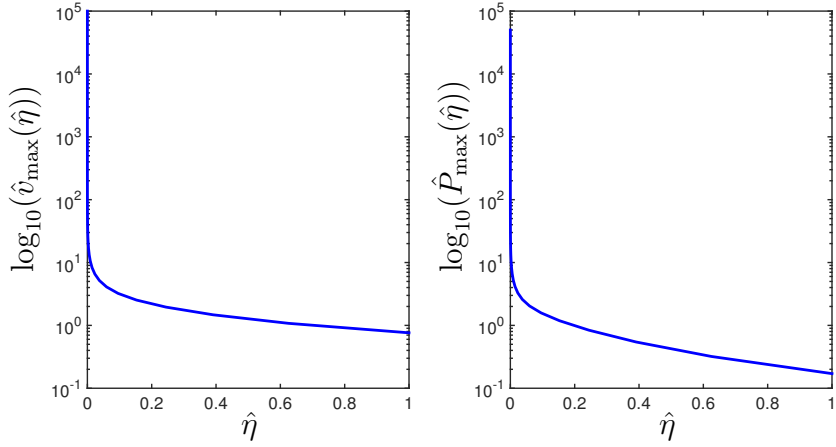

Clearly, the dimensionless parameter $\heta$ is the sole determinant of the blow-up in $\hat{v}_{\max }\left( \hat{\eta}\right) $ and $\hat{\mathcal{P}}_{\max }\left( \hat{\eta}\right) $, with smaller values of $\heta$ leading to higher values of $\hat{v}_{\max }\left( \hat{\eta}\right) $ and $\hat{\mathcal{P}}_{\max }\left( \hat{\eta}\right)$.
However, it is important to recall that $\heta$ does not depend solely on the viscoelasticity coefficient. The definition $\heta=\eta K_0/L^2$ implies that also small values of the permeability $K_0$ and/or large values of the domain dimension $L$ would reduce the value of $\heta$. Precisely,
going back to dimensional units, we can use~\eqref{eq:scalf_velocity} and~\eqref{eq:scalf_power_density}
to obtain the following expressions:
\begin{align}
& v_{\max}\left( \hat{\eta} \right) = \frac{P_{\text{ref}} K_0}{L} \hat{v}_{\max }\left( \hat{\eta}\right)
& \label{eq:v_max_dimensional} \\
& \mathcal{P}_{\max} \left( \hat{\eta} \right) = \frac{P_{\text{ref}}^2 K_0}{L}
\hat{\mathcal{P}}_{\max} \left( \hat{\eta} \right). & \label{eq:powerdensity_max_dimensional}
\end{align}

%
We can see that the magnitude $P_{\text{ref}}$ of the driving term appears as a multiplying constant in the expressions above, meaning that larger values of $P_{\text{ref}}$ will lead to larger values of $v_{\max}$ and $\mathcal{P}_{\max }$. However, the factor that controls the blow-up still remains the dimensionless parameter
$\heta = \eta K_0/L^2$.
These concepts are illustrated in Fig.~\ref{fig:dimensional_M_eta_step_pulse},
where we set $K_0/L = 1 \,\unit{m^2 s \,Kg^{-1}}$ and we consider increasing values of $P_{\text{ref}}$
in the range $[10^{-3}, 10^3] \, \unit{N m^{-2}}$.
\begin{figure}[h!]
\begin{center}
\begin{tikzpicture}
\node at (0,0)
{\includegraphics[scale=0.35]{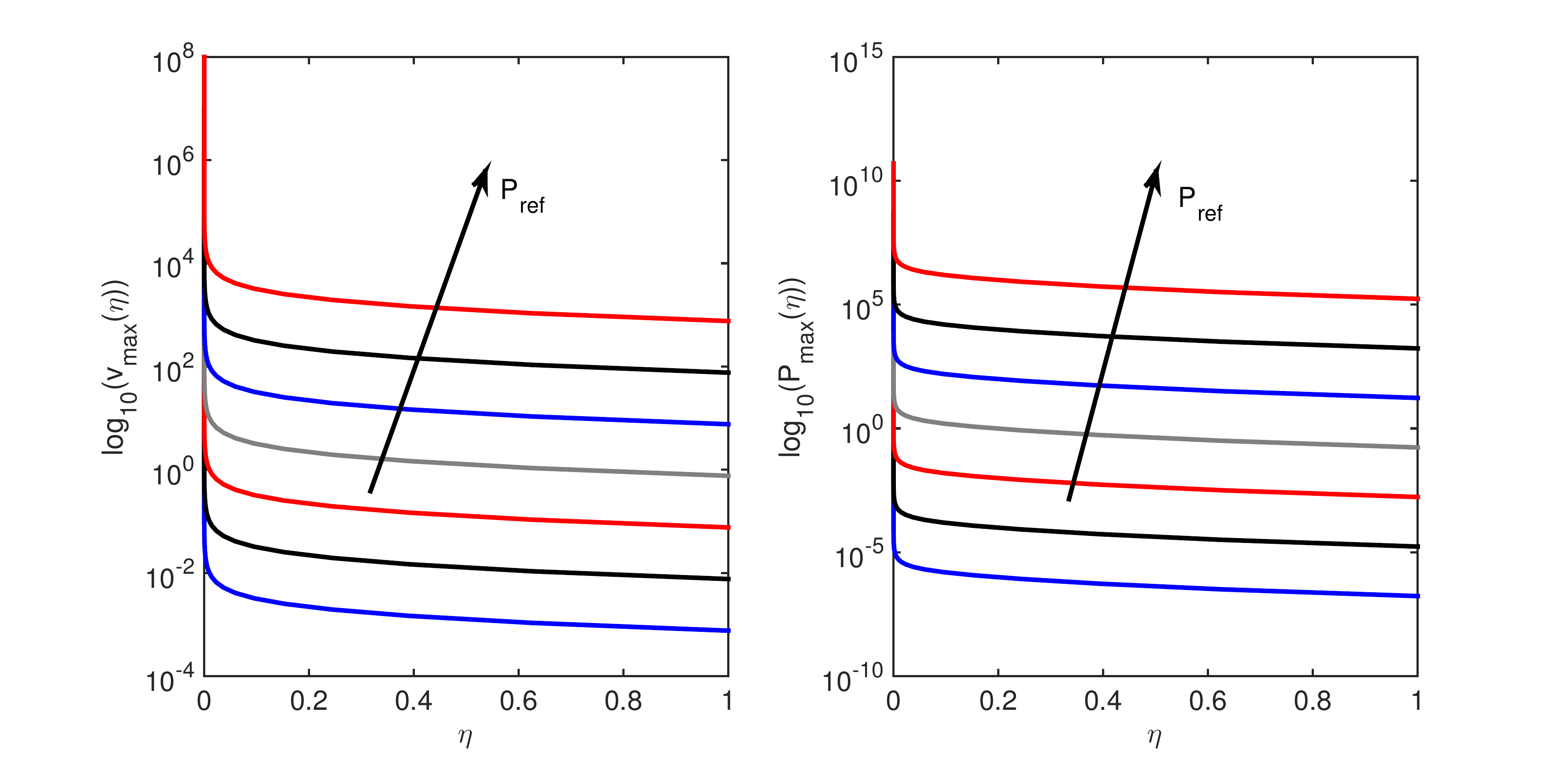}};
\fill [white] (-3,-3) rectangle (-2,-2.7);
\node [right] at (-2.8,-2.8) {$\hat{\eta}$};
\fill [white] (2.5,-3) rectangle (3.5,-2.7);
\node [right] at (2.8,-2.8) {$\hat{\eta}$};
\fill [white] (-5.6,-1) rectangle (-5.2,1);
\node [right,rotate=90] at (-5.4,-1) {$\log_{10}(\hat{v}_{\max}(\hat{\eta}))$};
\fill [white] (-0.3,-1) rectangle (0.3,1);
\node [right,rotate=90] at (0.1,-1) {$\log_{10}(\hat{P}_{\max}(\hat{\eta}))$};
\end{tikzpicture}
\caption{Left panel: dimensional maximal discharge velocity as a function of $\heta$.
Right panel: dimensional power density as a function of $\heta$. Log$_{10}$-scale is used on the $y$-axis to
better highlight blow-up of both quantities as $\hat{\eta} \rightarrow 0$. We set $K_0/L = 1 \,
\unit{m^2 s Kg^{-1}}$. The black arrows indicate increasing values of $P_{\text{ref}}$
in the range $[10^{-3}, 10^3] \, \unit{N m^{-2}}$.}
\label{fig:dimensional_M_eta_step_pulse}
\end{center}
\end{figure}
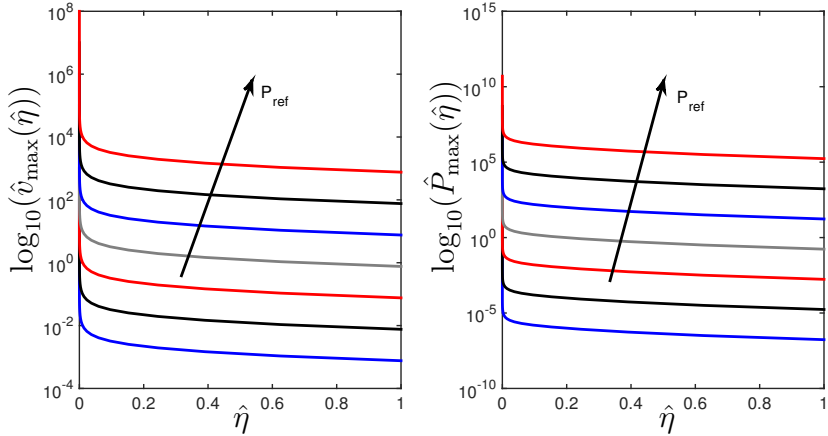

\section{The case $\hat{P}(\hatt)$ = trapezoidal pulse}\label{sec:trapezoidal_pulse}

Let us now consider the case of a driving term given by a trapezoidal pulse, where the signal switch on and switch off are characterized by linear ramps. Thus,
let $\hat{P}\left( \hat{t}\right) $
be the dimensionless trapezoidal pulse of unit amplitude and let
$\hat{\varepsilon}$, $\hat{\tau}>0$. Here we consider the following expression for  $\hat{P}\left( \hat{t}\right) $:
\begin{equation}\label{eq:P_trapezoidal_pulse}
\hat{P}\left( \hat{t}\right) =\left\{
\begin{array}{lll}
0 & \text{if} & \hat{t}<0
\\
\dfrac{\hat{t}}{\hat{\varepsilon}} & \text{if} & 0\leq \hat{t}<\hat{%
\varepsilon} \\
1 & \text{if} & \hat{\varepsilon}\leq \hat{t}<\hat{\varepsilon}+\hat{\tau}
\\
\dfrac{\hat{\tau}-\hat{t}}{\hat{\varepsilon}}+2 & \text{if} & \hat{%
\varepsilon}+\hat{\tau}\leq \hat{t}<2\hat{\varepsilon}+\hat{\tau} \\
0 & \text{if} & \hat{t}\geq 2\hat{\varepsilon}+\hat{\tau}%
\end{array}%
\right.
\end{equation}%
where $\hat{\varepsilon}$ is the pulse rise/fall time and $\hat{\tau}$ is
the pulse duration, as depicted in Fig.~\ref{fig:trapezoidal_pulse}.  Thus, the trapezoidal pulse defined in \eqref {eq:P_trapezoidal_pulse} reduces to a rectangular step pulse
as $\hat{\varepsilon}\rightarrow 0$.

\begin{figure}[h!]
\begin{center}
\begin{tikzpicture}
\node[below] at (-0.05,0) {$0$};
\draw[dashed] (-3,0) -- (-2,0);
\draw[->] (0,0) -- (7.2,0);
\node[below] at (7.1,0) {$\hatt$};
\draw[->] (0,0) -- (0,3.5);
\node[left] at (-0.05,3.3) {$\hat{P}(\hat{t})$};
\draw[very thick] (-2,0) -- (0,0);
\draw[very thick] (0,0) -- (1,2);
\draw[very thick] (1,2) -- (4,2);
\draw[very thick] (4,2) -- (5,0);
\draw[very thick] (5,0) -- (6.8,0);
\draw[dashed] (0,2) -- (1,2);
\node[left] at (-0.05,2) {$1$};
\draw[dashed] (1,0) -- (1,2);
\node[below] at (1,0) {$\hat{\varepsilon}$};
\draw[dashed] (4,0) -- (4,2);
\node[below] at (4,0) {$\hat{\tau}+\hat{\varepsilon}\quad$};
\node[below] at (5,0) {$\quad\hat{\tau}+2\hat{\varepsilon}$};
\end{tikzpicture}
\end{center}
\caption{Schematic representation of the dimensionless trapezoidal pulse $\hat{P}(\hatt)$ defined in~\eqref{eq:P_trapezoidal_pulse}. Here, the signal switch on and switch off are characterized by linear ramps.}
\label{fig:trapezoidal_pulse}
\end{figure}
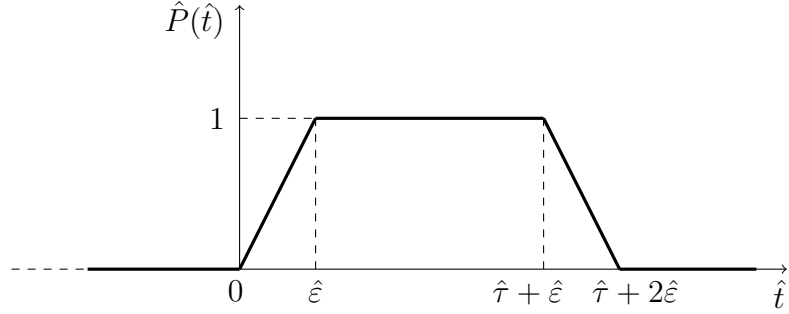

Let us now compute the dimensionless
discharge velocity resulting from the application of the trapezoidal pulse at the boundary, henceforth denoted by $\hat{V}_{\hat{\eta}}\left( \hat{x},\hat{t}%
\right) $ to distinguish it from the discharge velocity $\hat{v}_{\hat{\eta}}\left( \hat{x},\hat{t}%
\right) $ obtained for the step pulse. To this end, we observe that the trapezoidal pulse is actually the linear combination of four
ramp pulses of unit slope starting at times $\hat{t}=0$, $\hat{\varepsilon}$%
, $\hat{\varepsilon}+\hat{\tau}$,$\ 2\hat{\varepsilon}+\hat{\tau}$:%
\begin{eqnarray}\label{eq:trap_P}
\hat{P}\left( \hat{t}\right)  &=&\frac{1}{\hat{\varepsilon}}\left\{ \hat{t}%
H\left( \hat{t}\right) -\left( \hat{t}-\hat{\varepsilon}\right) H\left( \hat{%
t}-\hat{\varepsilon}\right) \right. \\
&&\left. -\left( \hat{t}-\hat{\tau}-\hat{\varepsilon}\right) H\left( \hat{t}-%
\hat{\tau}-\hat{\varepsilon}\right) +\left( \hat{t}-\hat{\tau}-2\hat{%
\varepsilon}\right) H\left( \hat{t}-\hat{\tau}-2\hat{\varepsilon}\right)
\right\} \notag
\end{eqnarray}%
where the function $H(\hatt)$ was defined in \eqref{eq:H}.
In addition, we can see that
\begin{equation*}
\hat{u}\left( \hat{x},\hat{t}\right) =\int_{0}^{\hat{t}}\hat{u}_{\hat{\eta}%
}\left( \hat{x},s\right) ds
\end{equation*}%
where $\hat{u}_{\hat{\eta}}\left( \hat{x},\hat{t}\right)$ is given by \eqref{u hat step}, solves problem \eqref{eq:1d_problem_per_hu} with $\hat{P}\left( \hat{t}\right) =\hat{t}H\left( \hat{t}\right)$
(the ramp pulse of unit slope starting at time $\hat{t}=0$).
Thus, using~\eqref{v}, we have that the discharge velocity is given by
$-\hat{u}_{\hat{\eta}}\left( \hat{x},\hat{t}\right)$, so that, applying the linear superposition principle, $\hat{V}_{\hat{\eta}}\left( \hat{x},\hat{t}\right) $ is the linear combination of the
velocities corresponding to the various components of the pulse, namely:%
\begin{eqnarray}
\hat{V}_{\hat{\eta}}\left( \hat{x},\hat{t}\right)  &=&\frac{1}{\hat{%
\varepsilon}}\left\{ -\hat{u}_{\hat{\eta}}\left( \hat{x},\hat{t}\right)
H\left( \hat{t}\right) +\hat{u}_{\hat{\eta}}\left( \hat{x},\hat{t}-\hat{%
\varepsilon}\right) H\left( \hat{t}-\hat{\varepsilon}\right) \right.
\label{V_max} \notag\\
&&+\hat{u}_{\hat{\eta}}\left( \hat{x},\hat{t}-\hat{\tau}-\hat{%
\varepsilon}\right) H\left( \hat{t}-\hat{\tau}-\hat{\varepsilon}\right) \\
&&\left. -
\hat{u}_{\hat{\eta}}\left( \hat{x},\hat{t}-\hat{\tau}-2\hat{\varepsilon}%
\right) H\left( \hat{t}-\hat{\tau}-2\hat{\varepsilon}\right) \right\}\,.
\notag
\end{eqnarray}
An illustration of the typical form of
$\hat{V}_{\hat{\eta}}\left( \hat{x},\hat{t}\right) $ for $\hat{\eta} \geq 0$ and $\hat{\varepsilon}>0$ is reported
in Fig.~\ref{fig:discharge_velocity_trapezoidal_pulse}.

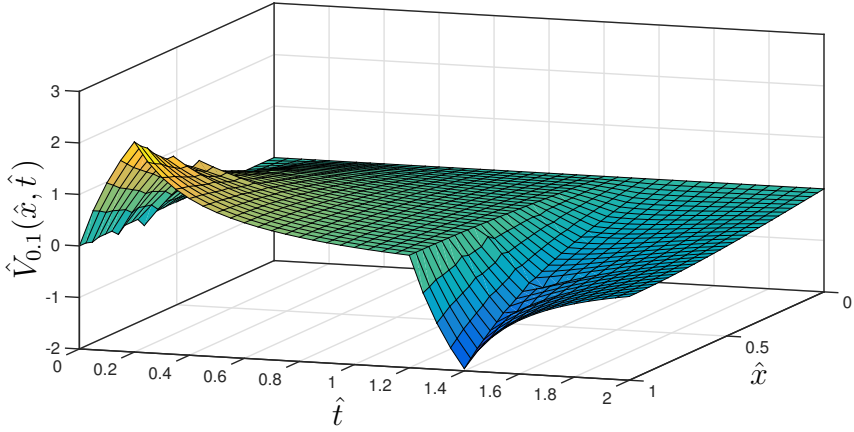
\begin{figure}[h!]
\begin{center}
\begin{tikzpicture}
\node at (0,0)
{\includegraphics[scale=0.35]{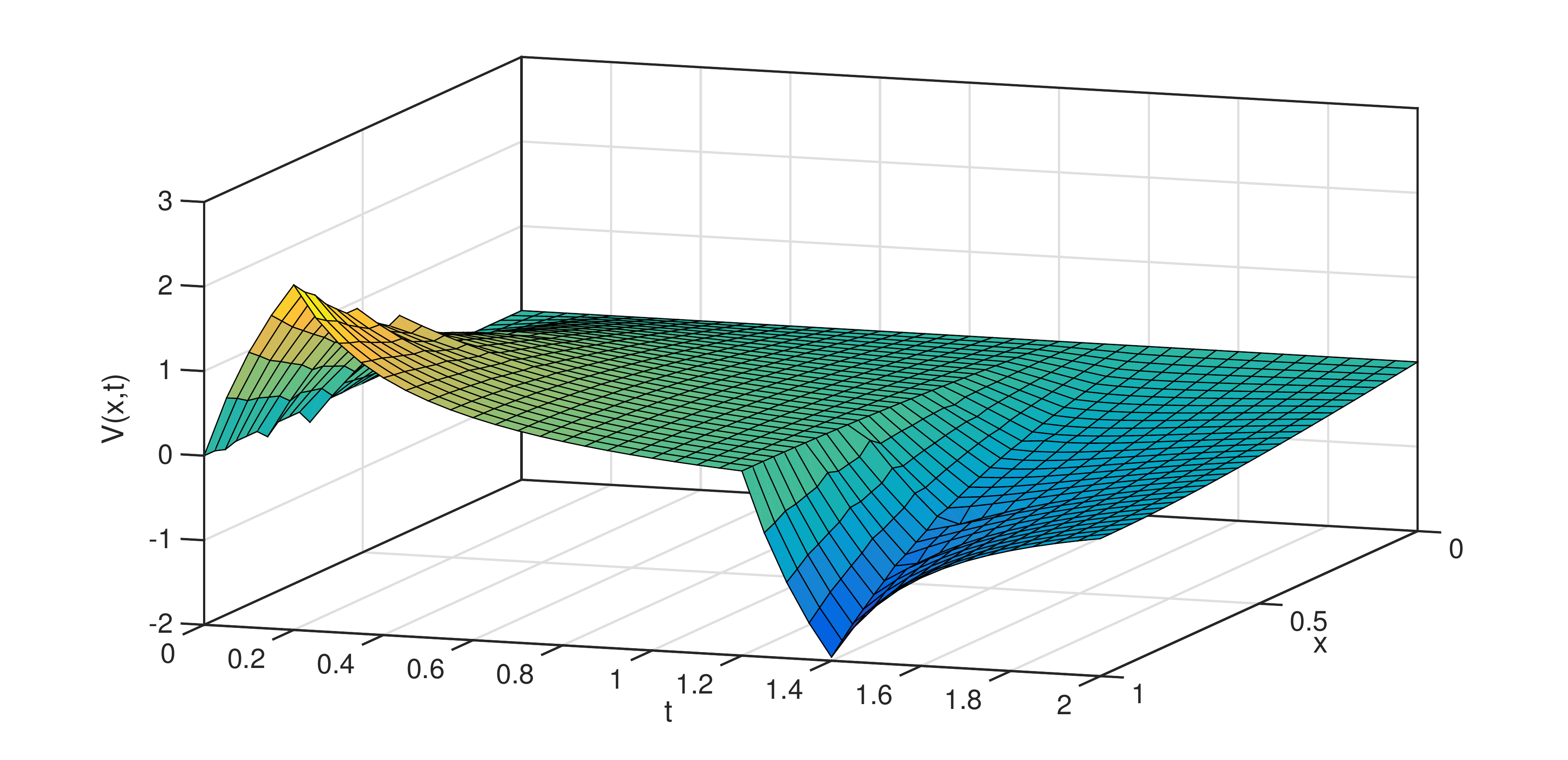}};
\fill [white] (4,-2.5) rectangle (4.5,-2.05);
\node [right] at (4,-2.3) {$\hat{x}$};
\fill [white] (-1.5,-3) rectangle (-0.7,-2.6);
\node [right] at (-1.5,-2.8) {$\hat{t}$};
\fill [white] (-5.6,-0.6) rectangle (-5.1,0.5);
\node [right,rotate=90] at (-5.4,-1.2) {$\hat{V}_{0.1}(\hat{x},\hat{t}\,)$};
\end{tikzpicture}
\caption{Dimensionless discharge velocity $\hat{V}_{\hat{\eta}}\left( \hat{x},\hat{t}%
\right) $ for $\hat{\eta}=0.1$, $\hat{\varepsilon}=0.2$, $\hat{\tau}=1$.} \label{fig:discharge_velocity_trapezoidal_pulse}
\end{center}
\end{figure}
The maximum possible discharge velocity occurs at $\hat{x}=1$, $%
\hat{t}=\hat{\varepsilon}$ and can be written as%
\begin{align}
\hat{V}_{\max }\left( \hat{\eta},\hat{\varepsilon}\right)
&=\max_{\substack{ %
0\leq \hat{x}\leq 1 \\ \hat{t}\geq 0}}\left\vert \hat{V}_{\hat{\eta}}\left(
\hat{x},\hat{t}\right) \right\vert =\hat{V}_{\hat{\eta}}\left( 1,\hat{%
\varepsilon}\right) \nonumber\\
&=\frac{2}{\hat{\varepsilon}}\sum_{n=0}^{\infty }\frac{1}{%
\hat{\lambda}_{n}}\left\{ 1-\exp \left( -\frac{\hat{\lambda}_{n}\hat{%
\varepsilon}}{1+\hat{\eta}\hat{\lambda}_{n}}\right) \right\} .\label{eq:V_max_eta_eps}
\end{align}%
The behavior of $\hat{V}_{\max } $ with respect to $\hat{\eta}$ and $\hat{\varepsilon}$ is reported in Fig.~\ref{fig:max_discharge_velocity_trapez_pulse}.
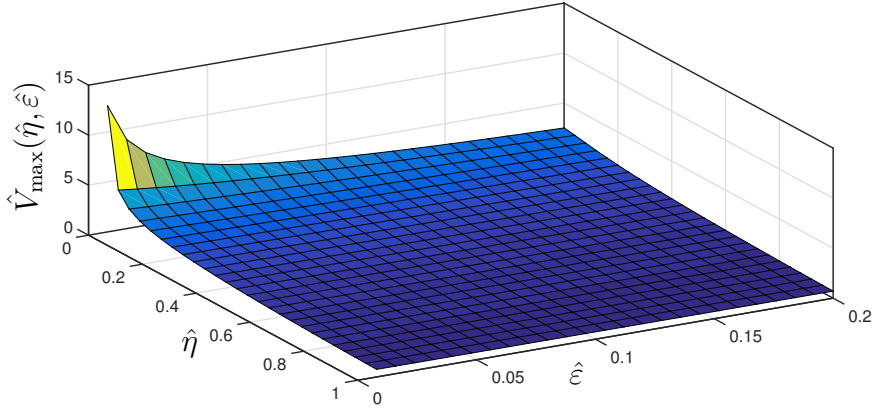
\begin{figure}[h!]
\begin{center}
\begin{tikzpicture}
\node at (0,0)
{\includegraphics[scale=0.35]{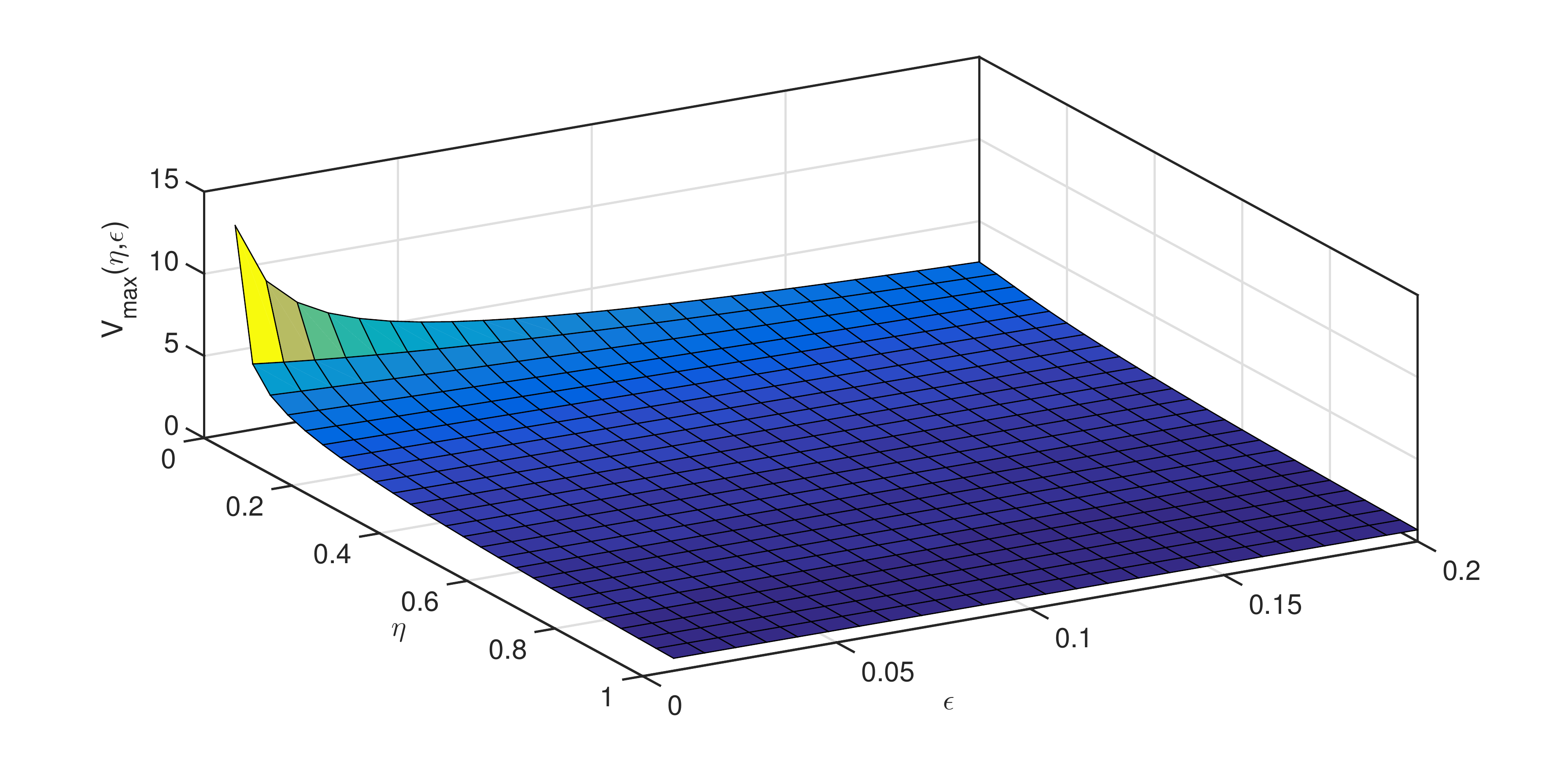}};
\fill [white] (1,-3) rectangle (1.5,-2.5);
\node [right] at (1.5,-2.3) {$\hat{\varepsilon}$};
\fill [white] (-3.3,-2.2) rectangle (-3,-1.9);
\node [right] at (-3.6,-1.9) {$\hat{\eta}$};
\fill [white] (-5.7,0) rectangle (-5.2,1.7);
\node [right,rotate=90] at (-5.5,-0.3) {$\hat{V}_{\max}(\hat{\eta},\hat{\varepsilon})$};
\end{tikzpicture}
\caption{Dimensionless maximal discharge velocity
$\hat{V}_{\max }\left( \hat{\eta},\hat{\varepsilon}\right)$ as a function of
$\hat{\eta}$ and $\hat{\varepsilon}$.}
\label{fig:max_discharge_velocity_trapez_pulse}
\end{center}
\end{figure}
Interestingly, for all $\hat{\eta}>0$ and $\hat{\varepsilon}\geq 0$ we have%
\begin{equation*}
\hat{V}_{\max }\left( \hat{\eta},\hat{\varepsilon}\right) \leq \hat{V}_{\max
}\left( \hat{\eta},0\right)
\end{equation*}%
and
\begin{equation*}
\hat{V}_{\max }\left( \hat{\eta},0\right) =\hat{v}_{\max }\left( \hat{\eta}%
\right)
\end{equation*}%
since the trapezoidal pulse reduces to a rectangular pulse as $\hat{%
\varepsilon}\rightarrow 0$. Thus, no blow-up takes place in the
viscoelastic case when the pulse is rectangular. Similarly, for all $\hat{%
\eta}\geq 0$ and $\hat{\varepsilon}>0$ it is%
\begin{equation*}
\hat{V}_{\max }\left( \hat{\eta},\hat{\varepsilon}\right) \leq \hat{V}_{\max
}\left( 0,\hat{\varepsilon}\right) =\frac{2}{\hat{\varepsilon}}%
\sum_{n=0}^{\infty }\frac{1-\exp \left( -\hat{\lambda}_{n}\hat{\varepsilon}%
\right) }{\hat{\lambda}_{n}}
\end{equation*}%
hence no blow-up takes place even in the purely elastic case when the pulse
is trapezoidal.

\section{Application in biomechanics: role of structural viscoelasticity in confined compression tests for biological tissues}\label{sec:athesian_soltz_case}

In this section, we utilize the mathematical analysis developed above to study some interesting features of confined compression tests, which are often utilized in biomechanics to characterize the properties of biological tissues.
A schematic of the confined compression experimental setting is depicted in Fig.~\ref{fig:confined_compression_chamber}, where a compressive load
is applied at the chamber top surface while
the bottom surface is maintained fixed.
Due to confinement,
deformation occurs only in the $x$ direction. The analogy with the 1D setting depicted in Fig.~\ref{fig:1D_model} is  straightforward, with a note of caution that 1D solutions are most representative of the biomechanical status of the material along vertical lines towards the center of the chamber.

\begin{figure}[h!]
\begin{center}
\begin{tikzpicture}
\node[inner sep=0pt] at (0,0)
    {\includegraphics[width=.4\textwidth]{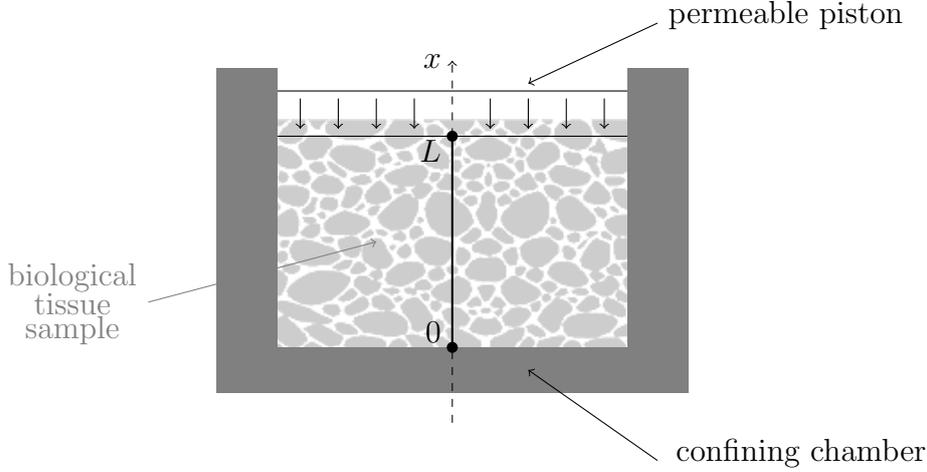}};
\draw (-3,2) rectangle (3,1.4);
\draw[->,black] (-2.5,1.9) -- (-2.5,1.5);
\draw[->,black] (-2,1.9) -- (-2,1.5);
\draw[->,black] (-1.5,1.9) -- (-1.5,1.5);
\draw[->,black] (-1,1.9) -- (-1,1.5);
\draw[->,black] (-0.5,1.9) -- (-0.5,1.5);
\draw[->,black] (2.5,1.9) -- (2.5,1.5);
\draw[->,black] (2,1.9) -- (2,1.5);
\draw[->,black] (1.5,1.9) -- (1.5,1.5);
\draw[->,black] (1,1.9) -- (1,1.5);
\draw[->,black] (0.5,1.9) -- (0.5,1.5);
\node[right] at (2.7,3) {permeable piston};
\draw[->] (2.7,2.9) -- (1,2.1);
\fill [gray] (-3.1,-2) rectangle (3.1,-1.4);
\fill [gray] (-3.1,-2) rectangle (-2.3,2.3);
\fill [gray] (2.3,-2) rectangle (3.1,2.3);
\node[right] at (2.8,-2.8) {confining chamber};
\draw[->] (2.7,-2.9) -- (1,-1.7);
\draw[black, thick] (0,-1.4) -- (0,1.4);
\draw[->,black, dashed] (0,1.4) -- (0,2.4);
\node[left, black] at (0,2.4) {$x$};
\draw[black, dashed] (0,-2.4) -- (0,-1.4);
\fill (0,1.4) circle (.4ex);
\fill (0,-1.4) circle (.4ex);
\node[left] at (0,-1.2) {$0$};
\node[left] at (0,1.2) {$L$};
\draw[->,gray] (-4,-0.8) -- (-1,0);
\node[above,gray] at (-5,-0.8) {biological};
\node[above,gray] at (-5,-1.1) {tissue};
\node[above,gray] at (-5,-1.5) {sample};
\end{tikzpicture}
\end{center}
\caption{Schematic representation of a confined compression chamber.}
\label{fig:confined_compression_chamber}
\end{figure}

The 1D model described in Section~\ref{sec:1D_problem_and_formulation} allows us to generalize the mathematical analysis carried out in~\cite{Soltz1998} by quantifying the effect of structural viscoelasticity on the tissue response to sudden changes in external pressure during confined compression experiments.
In this perspective, $L$ represents the thickness of the undeformed tissue
sample and $P_{\text{ref}}$ represents the magnitude of the total compressive stress
applied at $x=L$. It is important to notice that the
confined compression creep experiment described in~\cite{Soltz1998} is characterized by the same initial and boundary conditions as those given in~\eqref{eq:1D_BC_IC_s}.
However, the mathematical description adopted in~\cite{Soltz1998} (based on the theory developed
in~\cite{mow}), does not account for structural viscoelasticity, which amounts to setting
$\eta=0$ in~\eqref{eq:1D_model_sigma_0}.
Table~\ref{tab:atesian_soltz} summarizes the parameter values corresponding to the experiments on articular cartilage reported in~\cite{Soltz1998}.
\begin{table}[h!]
\centering
\begin{tabular}{|l|l|l|}
\hline
\textsf{symbol} & \textsf{value} & \textsf{units} \\ \hline
$L$             & $0.81 \cdot 10^{-3}$ & $\unit{m}$ \\ \hline
$\mu$           & $0.97 \cdot 10^6$    & $\unit{N m^{-2}}$ \\ \hline
$\eta$          & $0$                  & $\unit{N \,s \,m^{-2}}$ \\ \hline
$K_0$             & $2.9 \cdot 10^{-16}$ & $\unit{m^4 N^{-1} s^{-1}}$ \\ \hline
$P_{\text{ref}}$           & $6 \cdot 10^{4}$     & $\unit{N m^{-2}}$ \\ \hline
\end{tabular}
\caption{Numerical values of model parameters in the confined compression experiment for articular cartilage
reported in~\cite{Soltz1998}.}
\label{tab:atesian_soltz}
\end{table}

Characteristic velocities for the confined compression experiments reported in~\cite{Soltz1998} are of the order of 0.35 $\mu$m/s.
However, our analysis showed that velocities in the sample may attain much higher values, particularly near the permeable piston, should viscoelasticity be too low and/or the time profile of the load protocol be too sharp. In order to make these concepts more specific, let us assume that: (i) the load protocol for the confined compression experiment corresponds to the trapezoidal pulse described in Section~\ref{sec:trapezoidal_pulse}; and that (ii) we should ensure that discharge velocities always remain below the threshold value $V_{\text{th}}=0.35$ $\mu$m/s in order to preserve the sample from microstructural damage. This means that $\hat{\eta}$ and $\hat{\varepsilon}$ should be chosen as to guarantee that
\begin{equation}
\hat{V}_{\max } (\hat{\eta},\hat{\varepsilon}) <  \hat{V}_{\text{th}} =  \frac{L}{P_{\text{ref}} K_0}\, V_{\text{th}} = 16.3\,,
\end{equation}
where $\hat{V}_{\max } (\hat{\eta},\hat{\varepsilon})$ can be computed using the expression \eqref{eq:V_max_eta_eps}. Fig.~\ref{fig:max_discharge_velocity_trapez_pulse_AS} compares $ \hat{V}_{\text{th}}$ with the values of $\hat{V}_{\max } $ obtained for different values of $\hat{\eta}$ and $\hat{\varepsilon}$.  Indeed, we can identify a whole region in the $(\hat{\eta},\hat{\varepsilon}) -$plane for which $\hat{V}_{\max } (\hat{\eta},\hat{\varepsilon}) <  \hat{V}_{\text{th}}$.

\begin{figure}[h!]
\begin{center}
\begin{tikzpicture}
\node[inner sep=0pt] at (0,0)
{\includegraphics[scale=0.5]{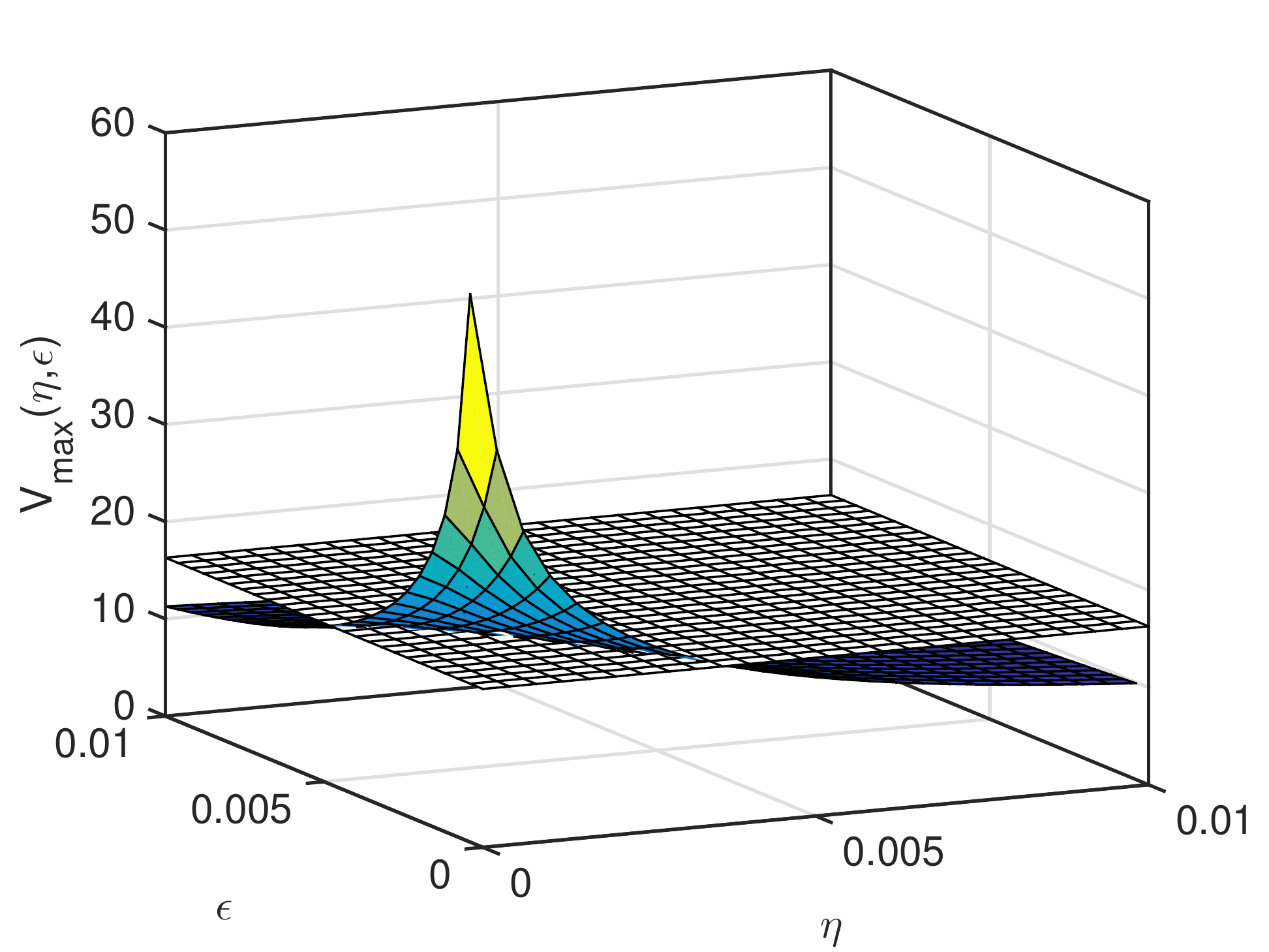}};
\fill [white] (0,-4) rectangle (3,-3.2);
\node [above] at (1,-3.5) {$\hat{\eta}$};
\fill [white] (-5,-4) rectangle (-2,-3.2);
\node [above] at (-2.5,-3.2) {$\hat{\varepsilon}$};
\fill [white] (-5,-1) rectangle (-4.3,1.2);
\node [right] at (4.8,-1) {$\hat{V}_{\max}(\hat{\eta},\hat{\varepsilon})$};
\draw [->] (4.8,-1) -- (3.6,-1.4);
\node [right] at (4.7,0) {$\hat{V}_{\text{th}}$};
\draw [->] (4.7,0) -- (3.4,-0.8);
\end{tikzpicture}
\end{center}
\caption{Comparison between the dimensionless maximum velocity $\hat{V}_{\max}(\hat{\eta},\hat{\varepsilon})$ obtained in the case of trapezoidal pulse using expression \eqref{eq:V_max_eta_eps} and the threshold velocity $\hat{V}_{\text{th}} = 16.3$ typical of confined compression experiments~\cite{Soltz1998}.}
\label{fig:max_discharge_velocity_trapez_pulse_AS}
\end{figure}
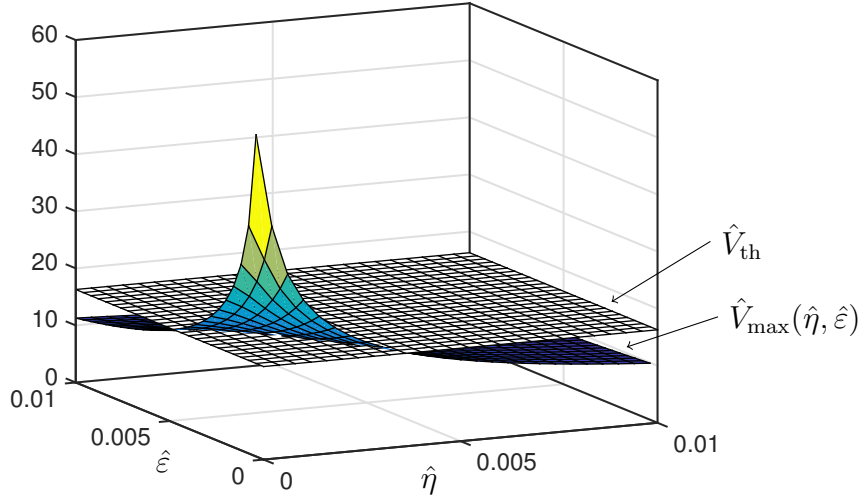

To better visualize the various regions of interest in the $(\hat{\eta},\hat{\epsilon}) -$plane, we reported  the colormap of the difference $ \hat{V}_{\text{th}}-\hat{V}_{\max }$ as a function of $\hat{\eta}$ and $\hat{\varepsilon}$ in Fig.~\ref{fig:vth-vmax}, where we clearly marked the curve in the parameter space for which $ \hat{V}_{\text{th}}=\hat{V}_{\max }$. These results show that, for example, choosing $\hat{\varepsilon} \geq 5 \cdot 10^{-3}$ would lead to fluid velocities below the recommended threshold regardless of the presence of structural viscoelasticity. On the other hand, should
the particular experimental conditions constrain us to enforce a rectangular pulse corresponding to $\hat{\varepsilon}=0$, we can still maintain the maximum velocity below the threshold by modifying the structural properties of the sample to have, for example, $\hat{\eta}\geq4\cdot 10^{-3}$.

\begin{figure}[h!]
\begin{center}
\begin{tikzpicture}
\node[inner sep=0pt] at (0,0)
{\includegraphics[scale=0.6]{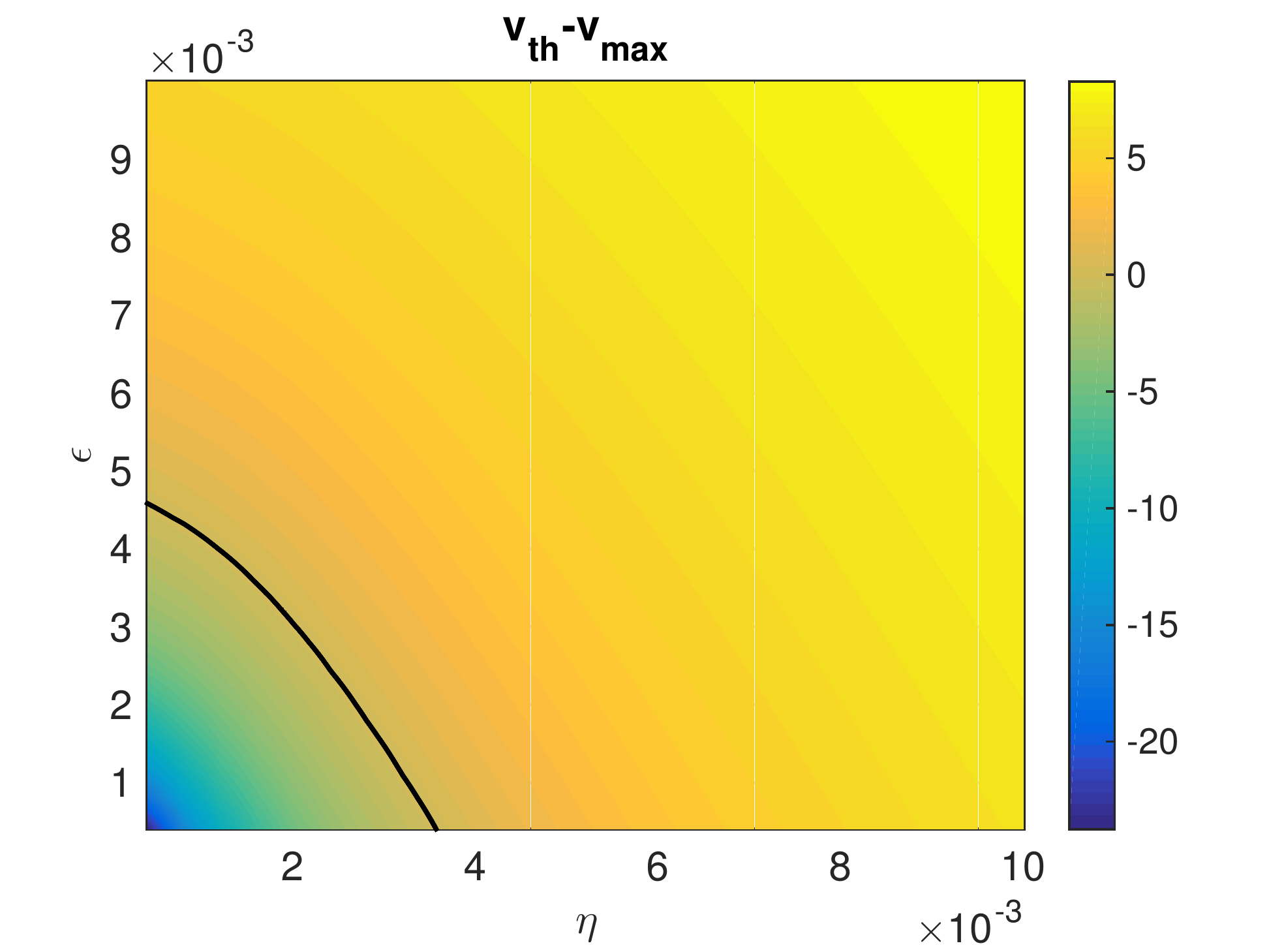}};
\fill [white] (-1,-4.5) rectangle (0,-4);
\node [above] at (-0.5,-4.5) {$\hat{\eta}$};
\fill [white] (-6,0) rectangle (-5,2);
\node [above] at (-5.5,0) {$\hat{\varepsilon}$};
\fill [white] (-2,3.8) rectangle (2,4.5);
\node [above] at (-0.2,3.7) {Colormap of $\hat{V}_{\text{th}}-\hat{V}_{\max }(\hat{\eta},\hat{\varepsilon})$};
\draw [fill=white] (-0.5,1) rectangle (2,1.6);
\node [above] at (0.8,1) {\small $\hat{V}_{\text{th}}>\hat{V}_{\max }(\hat{\eta},\hat{\varepsilon})$};
\begin{scope}[shift={(-5,-3.5)}]
\draw [fill=white] (-0.5,1) rectangle (2,1.6);
\node [above] at (0.8,1) {\small $\hat{V}_{\text{th}}<\hat{V}_{\max }(\hat{\eta},\hat{\varepsilon})$};
\end{scope}
\begin{scope}[shift={(-3.5,-5.7)}]
\draw [fill=white] (-0.5,1) rectangle (2,1.6);
\node [above] at (0.8,1) {\small $\hat{V}_{\text{th}}=\hat{V}_{\max }(\hat{\eta},\hat{\varepsilon})$};
\end{scope}
\draw [->,thick] (-2.6,-4.1) -- (-2.3,-2.7);
\end{tikzpicture}
\end{center}
\caption{Colormap of the difference $ \hat{V}_{\text{th}}-\hat{V}_{\max }$ as a function of $\hat{\eta}$ and $\hat{\varepsilon}$. As in Fig.~\ref{fig:max_discharge_velocity_trapez_pulse_AS}, $\hat{V}_{\max}(\hat{\eta},\hat{\varepsilon})$ is obtained  using expression \eqref{eq:V_max_eta_eps} and $\hat{V}_{\text{th}} = 16.3$ is the typical velocity of confined compression experiments~\cite{Soltz1998}.
The curve in the parameter space for which $ \hat{V}_{\text{th}}=\hat{V}_{\max }$ is reported in a thick black mark.}
\label{fig:vth-vmax}
\end{figure}
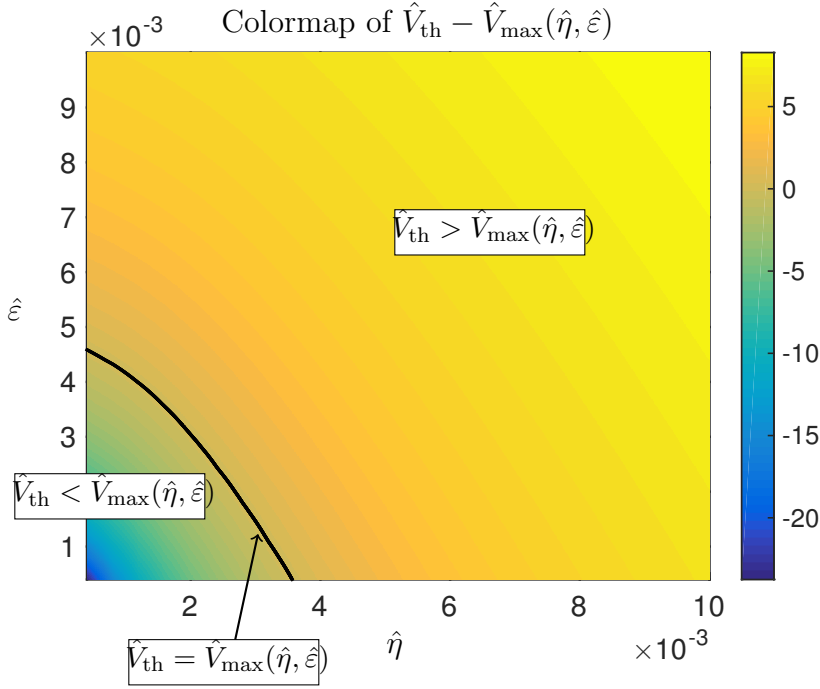

Interestingly, expression \eqref{eq:scalf_eta} defines  $\hat{\eta}= (K_0/L^2)\eta$, meaning that $\hat{\eta}$ can be increased by increasing the viscoelastic parameter $\eta$, increasing the permeability parameter $K_0$, or decreasing the sample thickness $L$. In the same spirit as~\cite{guidoboni2006}, a typical value of the viscoelastic parameter $\eta$ for biological tissues can be estimated by setting
\begin{subequations}\label{eq:creep_viscoelaticity}
\begin{align}
& \eta= \mu \tau_e & \label{eq:eta_creep}
\end{align}
where
\begin{align}
& \tau_e = L \sqrt{\frac{\rho}{\mu}} & \label{eq:elastic_time}
\end{align}
is the characteristic elastic time constant of the porous material under compression and $\rho$ is the mass density of the fluid component of the mixture. Since the densities of most of the biological fluids are very similar to that of water, let us set
$\rho=1000 \unit{\,Kg\, m^{-3}}$.
Substituting~\eqref{eq:eta_creep} and~\eqref{eq:elastic_time} in~\eqref{eq:scalf_eta} we obtain
\begin{align}
& \hat{\eta} = \frac{\tau_e}{[t]}= \frac{K_0}{L} \sqrt{\rho \mu}. & \label{eq:hat_eta_creep}
\end{align}
Replacing the parameter values of Table~\ref{tab:atesian_soltz} into~\eqref{eq:hat_eta_creep}
we obtain $\hat{\eta} = 1.15 \cdot 10^{-8}$, which is 5 orders of magnitude smaller than what is needed to attain fluid velocities below the recommended threshold in the case of a rectangular pulse in the pressure load. Thus, in order to control the maximum fluid velocity within the tissue sample, one may decide to:
(i) vary the properties of the biological sample by manipulating $\rho$, $\mu$, $\eta$, $K_0$ or $L$; and/or
(ii) vary the time profile of the load experimental protocol. The choice will depend on the particular conditions and constraints of the experiment at hand.

\end{subequations}

\section{Conclusions and future perspectives}\label{sec:conclusions}

The exact solutions obtained for the 1D models considered in this article allowed us to clearly identify a blow-up in the solution of certain poroelastic problems. The analysis allowed us to identify the main factors that give rise to the blow-up, namely the absence of structural viscoelasticity \textit{and} the time-discontinuity of the boundary source of traction. It is very important to emphasize that even a small viscoelastic contributions, namely $\heta\ll 1$, will prevent the blow-up. Similarly, a continuous time profile of the applied boundary load will prevent blow-up even if the structure is purely elastic.

These findings actually provide an evidence of a blow-up that was hypothesized by the theoretical work presented in~\cite{ARMA}. Interestingly, the a priori estimates derived in~\cite{ARMA} were not sufficient to bound the power density if the data were not regular enough. The 1D examples considered in this article show that indeed it is not possible to bound the power density if the structure is purely elastic and the boundary forcing term is discontinuous in time. Moreover, we have shown that this blow-up occurs even in the simple case of constant permeability.

In real situations, we will never see the fluid velocity spiking to infinity as predicted by the mathematical blow-up, since something will break first! For example, if the poroelastic model represents a biological tissue perfused by blood flowing in capillaries, as the maximum velocity becomes too high, capillaries will break letting blood out. Thus, from a practical perspective, it is crucial to identify parameters that can control the maximum value of the fluid velocity within a deformable porous medium in order to avoid microstructure damage. Our analysis showed that the maximum fluid velocity can be limited by:
(i) decreasing $P_{\text{ref}}$, thereby reducing the load on the medium;
(ii) increasing $\hat{\varepsilon}$, thereby smoothing the load action on the medium;
(iii) increasing $\eta$, thereby changing the structural properties of the solid component of the medium;
(iv) decreasing $K_0$, thereby changing the pore size, geometry and/or fluid viscosity;
(v) decreasing $L$, thereby changing the domain geometry. We have explored these concepts more concretely by applying them to the confined compression tests for biological tissues.

It is worth noting that
the extreme sensitivity of biological tissues to the active role of viscoelasticity
predicted by our theoretical analysis seems to agree with the conclusions of~\cite{Sack2009}
where experimental data based on magnetic resonance elastography show that viscoelasticity of
the brain is a result of structural alteration occurring in the course of physiological aging,
this suggesting that cerebral viscoelasticity may provide a sensitive marker
for a variety of neurological diseases such as normal pressure hydrocephalus,
Alzheimer's disease, or Multiple Sclerosis.

Further extensions of the current work may include:
(1) the study of porous deformable media with compressible components. This case mathematically 
corresponds to considering~\eqref{eq:fluid_content} instead of~\eqref{eq:fluid_content_incompressible}, 
and is of interest 
in the wider context of applications of the poro-visco-elastic model to problems in geomechanics
(cf.~the original contribution by Biot in~\cite{biot} and the more recent 
works~\cite{phillips,phillips2,phillips3} and~\cite{Barucq_1,Barucq_2}, 
devoted to computational analysis and theoretical investigations, respectively); (2)
the investigation of time discontinuities in the volumetric sources of linear momentum, which were identified by the analysis in~\cite{ARMA} as additional possible causes of blow-up. This case is actually extremely relevant in situations where the gravitational acceleration varies abruptly, such as during space flight take off and landing.

Finally, it is worth emphasizing that the 1D analysis presented in this article has made available a series of testable conditions leading to blow-up, and consequent microstructural damage, that could indeed be verified in laboratory experiments. We see the design and implementation of such experiments as the most interesting development of the present work.

\section*{Acknowledgements}
Dr. Bociu has been partially supported by NSF CAREER DMS-1555062. Dr. Guidoboni has been partially supported by the award NSF DMS-1224195, the Chair Gutenberg funds of the
Cercle Gutenberg (France) and the LabEx IRMIA (University of Strasbourg, France).
Dr. Sacco has been partially supported by Micron Semiconductor Italia S.r.l., SOW nr. 4505462139.


\bibliographystyle{plain}
\bibliography{biblio}

\end{document}